\documentclass[11pt]{amsart}

\usepackage{amsmath, amsfonts, amssymb,amsthm}
\usepackage{amstext}
\usepackage{mathrsfs}

\usepackage{float,epsf,subfigure}
\usepackage[all,cmtip]{xy}
\usepackage{hyperref}
\usepackage{algorithm}
\usepackage{algorithmic}
\usepackage{mathtools}
\usepackage{float}
\usepackage{bm}
\theoremstyle{plain}
\newtheorem{theorem}{Theorem}[section]

\newtheorem{proposition}[theorem]{Proposition}

\newtheorem{cor}[theorem]{Corollary}
\newtheorem{def-thm}[theorem]{Definition-Theorem}
\newtheorem{lemma}[theorem]{Lemma}

\theoremstyle{definition}

\def\min{\mathop{\mathrm{min}}}
\def\CC{\mathbb C}
\def\RR{\mathbb R}

\def\PP{\mathbb P}

\begin{document}
\title[H. Cartan's theory for Riemann surfaces]{H. Cartan's theory for Riemann surfaces}
\author[X.J. Dong]
{Xianjing Dong}

\address{Academy of Mathematics and Systems Sciences \\ Chinese Academy of Sciences \\ Beijing, 100190, P. R. China}
\email{xjdong@amss.ac.cn}


\date{}
\maketitle \thispagestyle{empty} \setcounter{page}{1}

\begin{abstract}
  We generalize the  H. Cartan's theory of holomorphic curves for  a general  open Riemann surface. Besides, 
a  vanishing theorem for  jet differentials and a Bloch's theorem  for  Riemann surfaces are   obtained.
\end{abstract}

\vskip\baselineskip

\setlength\arraycolsep{2pt}
\medskip

\section{Introduction}
\subsection{Main results}~

The value distribution theory \cite{Noguchi, ru} for holomorphic mappings has arisen strong research interest after R. Nevanlinna \cite{Nev} who founded two celebrated fundamental theorems for meromorphic functions in one complex variable in 1925.
  Many well-known results were obtained by extending  source manifolds and target manifolds, see H. Cartan \cite{Cart},  L. V. Ahlfors \cite{ahlfors}, Carlson-Griffiths \cite{gri}, Griffiths-King \cite{gri1}, E. I. Nochka \cite{nochka}, J. Noguchi \cite{Noguchi}, M. Ru \cite{ru, ru00, Ru-Sibony},  etc..  Particularly,  H. Cartan's theory for holomorphic curves is an  important branch of Nevanlinna theory, which was devised by
  H. Cartan in 1933. Let us review some key developments of this theory   related closely to our study of the  paper. 
  In 1970, H. Wu \cite{wu} gave a generalization of  H. Cartan's theory for the \emph{parabolic} Riemann surfaces (admitting a parabolic exhaustion function), see also  
 B. V. Shabat \cite{Shabat}.
  In 1983, E. I. Nochka \cite{nochka} established  the Second Main Theorem of holomorphic curves  into $\mathbb P^n(\mathbb C)$ intersecting  hyperplanes in $N$-subgeneral position.
 Recently in 2019,  Ru-Sibony \cite{Ru-Sibony} studied this theory for the discs in $\mathbb C.$
    
Due to the lack of  Green-Jensen   formula, 
 difficulties arise in extending  H. Cartan's theory to an arbitrary  Riemann surface by employing the classical approaches.
  So, 
we are  motivated by such  problem: What is  the H. Cartan's theory for Riemann surfaces?  
We aim to consider this problem following the stochastic technique initiated by T. K. Carne \cite{carne} and   developed by A. Atsuji \cite{at,atsuji}.
In this paper, we  give an extension of  H. Cartan's theory for a general Riemann surface. 
Note that the classical Second Main Theorem for  the discs was presented in terms of  incomplete metric induced from   Euclidean metric on $\mathbb C,$ see \cite{Ru-Sibony}.  
This incomplete metric produces a boundary term (main error term) in the Second Main Theorem, 
 and  hence it 
cannot describe 
  the value distribution  for holomorphic curves defined on  a general hyperbolic Riemann surface without boundary, 
  and  even it fails  for the  multi-connected 
    domains of $\mathbb C$ 
    though their  universal covering is the unit disc.  In the paper, we adopt a complete metric and  receive a Second Main Theorem of holomorphic curves defined on the  discs.      Instead of  the classical boundary condition, 
   we receive  a defect relation by means of a geometric condition.
    As applications, one also obtains a  vanishing theorem for  jet differentials as well as a Bloch's theorem for Riemann surfaces.
   In what follows, we introduce the main results. 
   
Let $S$ be an open (connected) Riemann surface. Due to the uniformization theorem,  one could equip $S$ with a complete  Hermitian metric $ds^2=2gdzd\bar{z}$   such that
 the 
   Gauss curvature
$K_S\leq0$  associated to $g,$  here $K_S$ is given  by
$$K_S=-\frac{1}{4}\Delta_S\log g=-\frac{1}{g}\frac{\partial^2\log g}{\partial z\partial\bar z}.$$
Clearly, $(S,g)$ is a complete K\"ahler manifold with  associated K\"ahler form
$$\alpha=g\frac{\sqrt{-1}}{2\pi}dz\wedge d\bar{z}.$$
\ \ \ \  Fix $o\in S$ as a reference point. Denote by $D(r)$  the geodesic disc centered at $o$ with radius $r,$ and by $\partial D(r)$  the boundary of $D(r).$
By Sard's theorem, $\partial D(r)$ is a submanifold of $S$ for almost all $r>0.$
Set
\begin{equation}\label{kappa}
  \kappa(t)=\min\big\{K_S(x): x\in \overline{D(t)}\big\},
\end{equation}
which is a non-positive, decreasing and continuous function  on $[0,\infty).$

 \begin{theorem}\label{thm} Let $f:S\rightarrow\mathbb P^n(\mathbb C)$ be a linearly non-degenerate holomorphic curve.
  Let $H_1,\cdots,H_q$ be hyperplanes of $\mathbb P^n(\mathbb C)$ in $N$-subgeneral position.
 Then
$$(q-2N+n-1)T_f(r)
 \leq \sum_{j=1}^qN_f^{[n]}(r,H_j)+O\Big(\log T_f(r)-\kappa(r)r^2+\log^+\log r\Big) \big{\|},
$$
 where $``\|"$ means  the  inequality holds outside a 
 subset of $r$ of  finite Lebesgue measure.
 \end{theorem}
 The term $\kappa(r)r^2$  comes  from  geometric nature   of  $S.$
 Let $S=\mathbb C,$  we obtain  $\kappa(r)\equiv0$ and $T_f(r)\geq O(\log r)$ as $r\rightarrow\infty.$ Thus, it covers  Nochka's result: 
\begin{cor}[Nochka, \cite{nochka}]\label{}
 Let $H_1,\cdots,H_q$ be hyperplanes of $\mathbb P^n(\mathbb C)$ in $N$-subgeneral position.
Let $f:\mathbb C\rightarrow\mathbb P^n(\mathbb C)$ be a linearly non-degenerate holomorphic curve.  Then
$$(q-2N+n-1)T_f(r)\leq \sum_{j=1}^qN_f^{[n]}(r,H_j)+O\big(\log T_f(r)\big)  \big{\|}.$$
\end{cor}

If $S$  is the Poincar\'e  disc $\mathbb D,$   then     
$\kappa(r)\equiv-1.$  We  can receive a more precise estimate $O(r)$ instead of $O(r^2).$ 
 \begin{cor}\label{ccc1}
 Let $f:\mathbb D\rightarrow\mathbb P^n(\mathbb C)$ be a linearly non-degenerate holomorphic curve.
 Let $H_1,\cdots,H_q$ be hyperplanes of $\mathbb P^n(\mathbb C)$ in $N$-subgeneral position.
 Then
$$(q-2N+n-1)T_f(r)\leq \sum_{j=1}^qN_f^{[n]}(r,H_j)+O\Big(\log T_f(r)+r\Big)  \big{\|}.$$
\end{cor}
The value distribution of holomorphic curves defined on a disc was studied by Ru-Sibony \cite{Ru-Sibony}. But, differently here, 
we  use a complete metric while Ru-Sibony used the incomplete one  induced from $\mathbb C.$ We here point out that our result  has the same accuracy with Ru-Sibony's one. 
  It's sufficient to compare the main error terms. Let $r(z)$ denote the poincar\'e distance of $z$ from $0$ for $z\in\mathbb D,$ and let $r, r_E$ be the Poincar\'e and Euclidean radius respectively. Note that 
  \begin{equation}\label{pop}
r(z)=\log\frac{1+|z|}{1-|z|}.
\end{equation}
In the classical case, the main  error term is 
$O(\log(1-r_E)^{-1}).$
By (\ref{pop})
$$\log\frac{1}{1-r_E}\leq r=\log\frac{1+r_E}{1-r_E}<\log\frac{1}{1-r_E}+\log2.$$
Therefore, the two main error terms are equivalent.

\begin{cor}[Defect relation]\label{defect1} Assume the same conditions as in Theorem $\ref{thm}$.
  If $f$ satisfies the growth condition
$$\liminf_{r\rightarrow\infty}\frac{\kappa(r)r^2}{T_f(r)}=0,$$
  then
  $$\sum_{j=1}^q\delta^{[n]}_f(H_j)\leq 2N-n+1.$$
  In particular, if $S$  is the Poincar\'e  disc, then  the  conclusion holds  only if 
  $$\limsup_{r\rightarrow\infty}\frac{r}{T_f(r)}=0.$$
\end{cor}

We introduce another two   results. Let $M$ be  a compact complex  manifold with   an ample divisor $A$ on $M.$
\begin{theorem}\label{thm2}
 Let $\omega$ be a logarithmic $k$-jet differential on $M$ which vanishes on $A.$
 Let $f:S\rightarrow M$ be a holomorphic curve
such that $f(S)$ is disjoint from the log-poles of $\omega.$ If $f$ satisfies  
  the growth condition
$$\liminf_{r\rightarrow\infty}\frac{\kappa(r)r^2}{T_{f,A}(r)}=0,$$
 then $f^*\omega\equiv0$ on $S.$   In particular, if $S$  is the Poincar\'e  disc,  then the conclusion holds only if 
  $$\limsup_{r\rightarrow\infty}\frac{r}{T_{f,A}(r)}=0.$$
\end{theorem}

Theorem \ref{thm2} gives an extension of the classical vanishing theorem  referred to Green-Griffiths \cite{g-g}, Siu-Yeung \cite{s-y}
and Demailly \cite{Dem}, etc..  Recently, Ru-Sibony \cite{Ru-Sibony}   proved a version for  discs  with incomplete metric.

Let $T$ be an $n$-dimensional complex torus with the universal covering $\CC^n$. The
standard coordinate $z=(z_1,\dots,z_m)$ of $\CC^n$ induces a global coordinate on $T$,
still denoted by $z$. Denote
$\beta:=dd^c\|z\|^2$ which is a positive $(1,1)$-form on $T$. As a consequence of Theorem \ref{thm2}, a generalized Bloch's theorem can be obtained
as follows
\begin{theorem}\label{thm4}
Let $f:S\to T$ be a
 holomorphic curve such that
\begin{equation}\label{th3}
\liminf_{r\to \infty}\frac{\kappa(r)r^2}{T_{f,\beta}(r)} = 0.
\end{equation}
 Then either $\overline{f(S)}$ is the translate of a subtorus of $T,$ or there exists a
variety $W$ of general type  and a mapping $R:\overline{f(S)}\to W$ such that $R\circ f$
does not satisfy $(\ref{th3})$. In particular, if $S$  is the Poincar\'e  disc,  then the conclusion holds only if 
  $$\limsup_{r\rightarrow\infty}\frac{r}{T_{f,\beta}(r)}=0.$$
\end{theorem}
\subsection{Brownian motions}~

 Let  $(M,g)$ be a Riemannian manifold with  Laplace-Beltrami operator $\Delta_M$ associated to  $g.$  For $x\in M,$ we denote by $B_x(r)$ the geodesic ball centered at $x$ with radius $r,$ and denote by $S_x(r)$ the geodesic sphere centered at $x$ with radius $r.$
 Apply  Sard's theorem, $S_x(r)$ is a submanifold of $M$ for almost all $r>0.$
A Brownian motion $X_t$ in $M$
is a heat diffusion  process  generated by $\Delta_M/2$ with transition density function $p(t,x,y)$  being the minimal positive fundamental solution of the  heat equation
  $$\frac{\partial}{\partial t}u(t,x)-\frac{1}{2}\Delta_{M}u(t,x)=0.$$
We denote by $\mathbb P_x$ the law of $X_t$ started at $x\in M$
 and by $\mathbb E_x$ the corresponding expectation with respect to $\mathbb P_x.$

 \noindent\textbf{A. Co-area formula}

  Let $D$  be a bounded domain with   smooth boundary $\partial D$ in $M$.
Fix $x\in D,$  we use $d\pi^{\partial D}_x$ to denote the harmonic measure  on $\partial D$ with respect to $x.$
This measure is a probability measure.
 Set
$$\tau_D:=\inf\{t>0:X_t\not\in D\}$$
which is a stopping time.
Denote by $g_D(x,y)$  the Green function of $\Delta_M/2$ for  $D$ with Dirichlet boundary condition and a pole at $x$, namely
$$-\frac{1}{2}\Delta_{M}g_D(x,y)=\delta_x(y), \ y\in D; \ \ g_D(x,y)=0, \ y\in \partial D,$$
where $\delta_x$ is the Dirac function.
 For $\phi\in \mathscr{C}_{\flat}(D)$
 (space of bounded continuous functions on $D$), the \emph{co-area formula} \cite{bass} asserts  that
$$ \mathbb{E}_x\left[\int_0^{\tau_D}\phi(X_t)dt\right]=\int_{D}g_{D}(x,y)\phi(y)dV(y).
$$
From Proposition 2.8 in \cite{bass}, we also have the relation of harmonic measures and hitting times that
\begin{equation}\label{w121}
  \mathbb{E}_x\left[\psi(X_{\tau_{D}})\right]=\int_{\partial D}\psi(y)d\pi_x^{\partial D}(y)
\end{equation}
for any $\psi\in\mathscr{C}(\overline{D})$.
Thanks to  the expectation $``\mathbb E_x",$ the co-area formula and (\ref{w121}) still work in the case when $\phi$ or $\psi$ has a pluripolar set of singularities.

 \noindent\textbf{B. Dynkin formula}

Let $u\in\mathscr{C}_\flat^2(M)$ (space of bounded $\mathscr{C}^2$-class functions on $M$), we have the famous \emph{It\^o formula} (see \cite{at,13, NN,itoo})
$$u(X_t)-u(x)=B\left(\int_0^t\|\nabla_Mu\|^2(X_s)ds\right)+\frac{1}{2}\int_0^t\Delta_Mu(X_s)dt, \ \ \mathbb P_x-a.s.$$
where $B_t$ is the standard  Brownian motion in $\mathbb R$ and $\nabla_M$ is gradient operator on $M$.
Take expectation of both sides of the above formula, it  follows \emph{Dynkin formula} (see \cite{at,itoo})
$$ \mathbb E_x[u(X_T)]-u(x)=\frac{1}{2}\mathbb E_x\left[\int_0^T\Delta_Mu(X_t)dt\right]
$$
for a stopping time $T$ such that each term  makes sense.
Notice that  Dynkin formula still holds for  $u\in\mathscr{C}^2(M)$ if $T=\tau_D.$
In further, it also works when $u$ is of a pluripolar set of singularities, particularly, for a plurisubharmonic function $u.$ 

\section{First Main Theorem}

Let $(S,g)$ be an open Riemann surface with  K\"ahler form $\alpha$ associated to the Hermitian metric $g.$
 Fix $o\in S$ as a reference point, we
  denote by $g_r(o,x)$ the Green function of $\Delta_S/2$ for $D(r)$ with Dirichlet boundary condition and a pole at $o,$ and by $d\pi^r_o(x)$ the harmonic
measure on $\partial D(r)$ with respect to $o.$
Let $X_t$ be the Brownian
motion with generator $\Delta_S/2$ starting from $o\in S.$ Moreover, we set the stopping time $$\tau_r=\inf\big\{t>0: X_t\not\in D(r)\big\}.$$
\subsection{Nevanlinna's functions}~

Let $$f:S\rightarrow M$$  be a holomorphic curve into a compact complex  manifold $M.$ We  introduce  the generalized  Nevanlinna's functions over 
  Riemann surface $S.$
Let $L\rightarrow M$
be an ample  holomorphic line bundle   equipped with Hermitian metric $h.$   We define the \emph{characteristic function} of $f$ with respect to $L$ by
 \begin{eqnarray*}\label{}
   T_{f,L}(r)
   &=& \pi\int_{D(r)}g_r(o,x)f^*c_1(L,h) \\
   &=&-\frac{1}{4}\int_{D(r)}g_r(o,x)\Delta_S\log h\circ f(x)dV(x),
 \end{eqnarray*}
 where $dV(x)$ is the Riemannian volume measure of $S.$ It can be easily known that $T_{f,L}(r)$ is independent
 of the choices of  metrics on $L,$ up to a bounded term.  
 Since a holomorphic line bundle 
 can be represented  as the difference of two  ample holomorphic line bundles,  the definition of $T_{f,L}(r)$ can  extend to
 an arbitrary holomorphic line bundle.
 Apply co-area formula, we have
 $$T_{f,L}(r)=-\frac{1}{4}\mathbb E_o\left[\int_0^{\tau_r}\Delta_S\log h\circ f(X_t)dt\right].$$
For simplicity,  
  we  use  notation  $T_{f,D}(r)$ to stand for   $T_{f,L_D}(r)$ for a   divisor $D$ on $M.$ 
  Similarly, a divisor can also  be written  as the difference of two ample divisors.
  The Weil function of  $D$  is well defined  by
$$\lambda_D(x)=-\log\|s_D(x)\|$$
up to a bounded term, here $s_D$ is the canonical section associated to $D.$   We   define the \emph{proximity function} of $f$ with respect to $D$  by
 $$m_f(r,D)=\int_{\partial D(r)}\lambda_D\circ f(x)d\pi^r_o(x).$$
 A relation between harmonic measures and hitting times shows that
 $$m_f(r,D)=\mathbb E_o\big[\lambda_D\circ f(X_{\tau_r})\big].$$
Locally, we write $s_D=\tilde{s}_De,$ where $e$ is a local holomorphic frame of $(L_D, h).$
 The \emph{counting function}
 of $f$ with respect to $D$ is defined by
\begin{eqnarray*}
N_f(r,D)
&=& \pi \sum_{x\in f^*D\cap D(r)}g_r(o,x) \\
&=& \pi\int_{D(r)}g_r(o,x)dd^c\big{[}\log|\tilde{s}_D\circ f(x)|^2\big{]} \\
&=&\frac{1}{4}\int_{D(r)}g_r(o,x)\Delta_S\log|\tilde{s}_D\circ f(x)|^2dV(x)
\end{eqnarray*}
in the sense of distributions or currents.

 Our definition of Nevanlinna's functions in above is  natural. When $S=\mathbb C,$ the Green function is $(\log\frac{r}{|z|})/\pi$
 and the harmonic measure is $d\theta/2\pi.$ So, by integration by part,  we  observe that
they agree with the classical ones.
\subsection{First Main Theorem}~

With the previous preparatory work, we are ready to prove the First Main Theorem of a holomorphic curve $f:S\rightarrow M$
 such that $f(o)\not\in {\rm{Supp}}D,$ where $D$ is a divisor on $M.$
Apply Dynkin formula to $\lambda_D\circ f(x),$  it yields that 
 $$\mathbb E_o\big[\lambda_D\circ f(X_{\tau_r})\big]-\lambda_D\circ f(o)=\frac{1}{2}
 \mathbb E_o\left[\int_0^{\tau_r}\Delta_S\lambda_D\circ f(X_t)dt\right].$$
The first term on the left hand side of the above equality is equal to $m_f(r,D),$ and the term on
the right hand side equals
$$\frac{1}{2}\mathbb E_o\left[\int_0^{\tau_r}\Delta_S\lambda_D\circ f(X_t)dt\right]=
\frac{1}{2}\int_{D(r)}g_r(o,x)\Delta_S\log\frac{1}{\|s_D\circ f(x)\|}dV(x)$$
due to co-area formula. Since $\|s_D\|^2=h|\tilde{s}_D|^2,$ where $h$ is a Hermitian metric 
on $L_D,$ then we get
\begin{eqnarray*}
\frac{1}{2}\mathbb E_o\left[\int_0^{\tau_r}\Delta_S\lambda_D\circ f(X_t)dt\right] &=&
-\frac{1}{4}\int_{D(r)}g_r(o,x)\Delta_S\log h\circ f(x)dV(x) \\
&&-\frac{1}{4}\int_{D(r)}g_r(o,x)\Delta_S\log|\tilde{s}_D\circ f(x)|^2dV(x) \\
&=& T_{f, D}(r)-N_f(r,D).
\end{eqnarray*}

Therefore, we obtain the First Main Theorem:
\begin{theorem}\label{first}  
 Let $f:S\rightarrow M$ be a holomorphic
curve with $f(o)\not\in {\rm{Supp}}D.$
  Then
  $$T_{f, D}(r)=m_f(r,D)+N_f(r,D)+O(1).$$
\end{theorem}
\begin{cor}[Nevanlinna's inequality]\label{bzda} We have
$$N_f(r,D)\leq T_{f, D}(r)+O(1).$$
\end{cor}
\section{Two key lemmas}

Let $(S,g)$ be a simply-connected and complete open Riemann surface with  Gauss curvature  $K_S\leq0$ associated to $g.$  Since uniformization theorem,  there is  a nowhere-vanishing holomorphic  vector field $\mathfrak X$ over $S.$
 In fact, it is known  from \cite{GR},   
 a nowhere-vanishing holomorphic vector field  always exists for  an arbitrary open Riemann surface.
\subsection{Calculus Lemma}~

Let $\kappa$ be defined by (\ref{kappa}). As is noted before,
$\kappa$ is a non-positive, decreasing continuous function  on $[0,\infty).$
  Associate the ordinary differential equation
  \begin{equation}\label{G}
    G''(t)+\kappa (t)G(t)=0; \ \ \ G(0)=0, \ \ G'(0)=1.
  \end{equation}
 We compare (\ref{G})  with $y''(t)+\kappa(0)y(t)=0$ under the same  initial conditions,
 $G$ can be easily estimated  as
$$G(t)=t \ \ \text{for}  \ \kappa\equiv0; \ \ \ G(t)\geq t \ \ \text{for} \ \kappa\not\equiv0.$$
This implies that
\begin{equation}\label{vvvv}
  G(r)\geq r \ \ \text{for} \ r\geq0; \ \ \ \int_1^r\frac{dt}{G(t)}\leq\log r \ \ \text{for} \ r\geq1.
\end{equation}
On the other hand, we  rewrite (\ref{G}) as the form
$$\log'G(t)\cdot\log'G'(t)=-\kappa(t).$$
Since $G(t)\geq t$ is increasing,
then the decrease and non-positivity of $\kappa$ imply that for each fixed $t,$ $G$  must satisfy one of the following two inequalities
$$\log'G(t)\leq\sqrt{-\kappa(t)} \ \ \text{for} \ t>0; \ \ \ \log'G'(t)\leq\sqrt{-\kappa(t)} \ \ \text{for} \ t\geq0.$$
Since $G(t)\rightarrow0$ as $t\rightarrow0,$ by integration, $G$ is bounded from above by
\begin{equation}\label{v2}
  G(r)\leq r\exp\big(r\sqrt{-\kappa(r)}\big) \ \  \text{for} \ r\geq0.
\end{equation}

 The main result of this subsection is the following Calculus Lemma:
\begin{theorem}[Calculus Lemma]\label{cal}
 Let $k\geq0$ be a locally integrable  function on $S$ such that it is locally bounded at $o\in S.$
 Then for any $\delta>0,$ there exists a constant $C>0$ independent of $k,\delta,$ and a subset $E_\delta\subseteq(1,\infty)$ of finite Lebesgue measure such that
$$
\mathbb E_o\big{[}k(X_{\tau_r})\big{]}
\leq \frac{F(\hat{k},\kappa,\delta)e^{r\sqrt{-\kappa(r)}}\log r}{2\pi C}\mathbb E_o\left[\int_0^{\tau_{r}}k(X_{t})dt\right]
$$
  holds for $r>1$ outside $E_\delta,$  where  $\kappa$ is defined by $(\ref{kappa})$ and $F$ is defined by
$$
F\big{(}\hat{k},\kappa, \delta\big{)}
=\Big\{\log^+\hat{k}(r)\cdot\log^+\Big(re^{r\sqrt{-\kappa(r)}}\hat{k}(r)\big\{\log^{+}\hat{k}(r)\big\}^{1+\delta}\Big)\Big\}^{1+\delta} \ \ \ \
$$with
$$\hat k(r)=\frac{\log r}{C}\mathbb E_o\left[\int_0^{\tau_{r}}k(X_{t})dt\right].$$
Moreover, we have the estimate
$$\log F(\hat{k},\kappa,\delta)
\leq O\Big(\log^+\log \mathbb E_o\left[\int_0^{\tau_{r}}k(X_{t})dt\right]+\log^+r\sqrt{-\kappa(r)}+\log^+\log r\Big).
$$
\end{theorem}

To prove  theorem \ref{cal}, we need to prepare some lemmas.
\begin{lemma}[\cite{atsuji}]\label{zz} Let $\eta>0$ be a constant. Then there is  a constant $C>0$ such that for
$r>\eta$ and $x\in B_o(r)\setminus \overline{B_o(\eta)}$
  $$g_r(o,x)\int_{\eta}^r\frac{dt}{G(t)}\geq C\int_{r(x)}^r\frac{dt}{G(t)}$$
 holds,  where  $G$ is defined by {\rm{(\ref{G})}}.
\end{lemma}

\begin{lemma}[\cite{ru}]\label{cal1} Let $T$ be a strictly positive nondecreasing  function of $\mathscr{C}^1$-class on $(0,\infty).$ Let $\gamma>0$ be a number such that $T(\gamma)\geq e,$ and $\phi$ be a strictly positive nondecreasing function such that
$$c_\phi=\int_e^\infty\frac{1}{t\phi(t)}dt<\infty.$$
Then, the inequality
  $$T'(r)\leq T(r)\phi(T(r))$$
holds for all $r\geq\gamma$ outside a subset of Lebesgue measure not exceeding $c_\phi.$ In particular, take  $\phi(t)=\log^{1+\delta}t$ for a number  $\delta>0,$  we have
  $$T'(r)\leq T(r)\log^{1+\delta}T(r)$$
holds for all $r>0$ outside a subset $E_\delta\subseteq(0,\infty)$ of finite Lebesgue measure.
\end{lemma}
We are ready to prove Theorem $\ref{cal}:$ 
\begin{proof}
We follow the arguments of Atsuji \cite{atsuji}. The simple-connectedness and non-positivity of  Gauss curvature of $S$ imply the following relation  (see \cite{Deb})
$$d\pi^r_{o}(x)\leq\frac{1}{2\pi r}d\sigma_r(x),$$
where  $d\sigma_r(x)$ is the induced volume measure on $\partial D(r).$
By  Lemma \ref{zz} and (\ref{vvvv}), we have
\begin{eqnarray*}
   \mathbb E_o\left[\int_0^{\tau_{r}}k(X_{t})dt\right] &=&
   \int_{D(r)}g_r(o,x)k(x)dV(x) \\ &=&\int_0^rdt\int_{\partial D(t)}g_r(o,x)k(x)d\sigma_t(x)  \\
&\geq& C\int_0^r\frac{\int_t^rG^{-1}(s)ds}{\int_1^rG^{-1}(s)ds}dt\int_{\partial D(t)}k(x)d\sigma_t(x) \\
&\geq& \frac{C}{\log r}\int_0^rdt\int_t^r\frac{ds}{G(s)}\int_{\partial D(t)}k(x)d\sigma_t(x), \\
 \mathbb E_o\big[k(X_{\tau_r})\big]&=&\int_{\partial D(r)}k(x)d\pi_o^r(x)\leq\frac{1}{2\pi r}\int_{\partial D(r)}k(x)d\sigma_r(x).
\end{eqnarray*}
Hence,
  \begin{eqnarray}\label{fr}
\mathbb E_o\left[\int_0^{\tau_{r}}k(X_{t})dt\right]&\geq&\frac{C}{\log r}\int_0^rdt\int_t^r\frac{ds}{G(s)}\int_{\partial D(o,t)}k(x)d\sigma_t(x), \nonumber \\
  \mathbb E_o\big[k(X_{\tau_r})\big]&\leq&\frac{1}{2\pi r}\int_{\partial D(r)}k(x)d\sigma_r(x).
\end{eqnarray}
 Set
$$\Lambda(r)=\int_0^rdt\int_t^r\frac{ds}{G(s)}\int_{\partial D(t)}k(x)d\sigma_t(x).$$
We conclude that
\begin{equation*}
 \Lambda(r)\leq\frac{\log r}{C}\mathbb E_o\left[\int_0^{\tau_{r}}k(X_{t})dt\right]=\hat{k}(r).
\end{equation*}
Since
$$\Lambda'(r)=\frac{1}{G(r)}\int_0^rdt\int_{\partial D(t)}k(x)d\sigma_t(x),$$
then it yields from (\ref{fr}) that
\begin{equation*}
  \mathbb E_o\big{[}k(X_{\tau_r})\big{]}\leq\frac{1}{2\pi r}\frac{d}{dr}\left(\Lambda'(r)G(r)\right).
\end{equation*}
By Lemma \ref{cal1} twice and (\ref{v2}), then for any $\delta>0$
\begin{eqnarray*}
   && \frac{d}{dr}\left(\Lambda'(r)G(r)\right) \\
&\leq& G(r)\Big\{\log^+\Lambda(r)\cdot\log^+\left(G(r)\Lambda(r)\big\{\log^+\Lambda(r)\big\}^{1+\delta}\right)\Big\}^{1+\delta}\Lambda(r)  \\
&\leq& re^{r\sqrt{-\kappa(r)}}\Big\{\log^+\hat k(r)\cdot\log^+\Big(re^{r\sqrt{-\kappa(r)}}\hat k(r)\big\{\log^+\hat k(r)\big\}^{1+\delta}\Big)\Big\}^{1+\delta} \hat k(r) \\
&=& \frac{F\big{(}\hat k,\kappa,\delta\big{)}re^{r\sqrt{-\kappa(r)}}\log r}{C}\mathbb E_o\left[\int_0^{\tau_{r}}k(X_{t})dt\right] \ \ \
\end{eqnarray*}
holds outside a set $E_\delta\subseteq(1,\infty)$ of finite Lebesgue measure.
Thus,
\begin{eqnarray*}
\mathbb E_o\big{[}k(X_{\tau_r})\big{]}
&\leq&\frac{F\big{(}\hat k,\kappa,\delta\big{)}e^{r\sqrt{-\kappa(r)}}\log r}{2\pi C}\mathbb E_o\left[\int_0^{\tau_{r}}k(X_{t})dt\right].
\end{eqnarray*}
Hence, we get the desired  inequality.
Indeed, for $r>1$  we compute that
$$
\log F(\hat k,\kappa,\delta)
\leq O\Big(\log^+\log^+\hat k(r)+\log^+r\sqrt{-\kappa(r)}+\log^+\log r\Big)$$
with
\begin{eqnarray*}
  \log^+\hat k(r)&\leq& \log \mathbb E_o\left[\int_0^{\tau_{r}}k(X_{t})dt\right]+\log^+\log r+O(1).
  \end{eqnarray*}
 Therefore, we have arrived at the required estimate.
 \end{proof}
\subsection{Logarithmic Derivative Lemma}~

Let $\psi$ be a meromorphic function on $(S,g).$
The norm of the gradient of $\psi$ is defined by
$$\|\nabla_S\psi\|^2=\frac{1}{g}\left|\frac{\partial\psi}{\partial z}\right|^2$$
in a local holomorphic coordinate $z.$ Locally, we can write $\psi=\psi_1/\psi_0,$ where $\psi_0,\psi_1$ are local holomorphic functions without common zeros. Regard $\psi$  as a holomorphic mapping into $\mathbb P^1(\mathbb C)$  by
$x\mapsto[\psi_0(x):\psi_1(x)].$
We define
$$
T_\psi(r)=\frac{1}{4}\int_{D(r)}g_r(o,x)\Delta_S\log\big{(}|\psi_0(x)|^2+|\psi_1(x)|^2\big{)}dV(x)$$
and
$$T(r,\psi):=m(r,\psi)+N(r,\psi),
$$
where
\begin{eqnarray*}
m(r,\psi)&=&\int_{\partial D(r)}\log^+|\psi(x)|d\pi^r_o(x), \\
N(r,\psi)&=&\pi \sum_{x\in \psi^{-1}(\infty)\cap D(r)}g_r(o,x).
\end{eqnarray*}
Let
  $i:\mathbb C\hookrightarrow\mathbb P^1(\mathbb C)$ be an inclusion defined by
 $z\mapsto[1:z].$  Via the pull-back by $i,$ we have a (1,1)-form $i^*\omega_{FS}=dd^c\log(1+|\zeta|^2)$ on $\mathbb C,$
 where $\zeta:=w_1/w_0$ and $[w_0:w_1]$ is
the homogeneous coordinate system of $\mathbb P^1(\mathbb C).$ The characteristic function of $\psi$ with respect to $i^*\omega_{FS}$ is defined by
$$\hat{T}_\psi(r) = \frac{1}{4}\int_{D(r)}g_r(o,x)\Delta_S\log(1+|\psi(x)|^2)dV(x).$$
Clearly, $\hat{T}_\psi(r)\leq T_\psi(r).$
We adopt the spherical distance $\|\cdot,\cdot\|$ on  $\mathbb P^1(\mathbb C),$ the proximity function of $\psi$  with respect to
$a\in \mathbb P^1(\mathbb C)$
is defined by
$$\hat{m}_\psi(r,a)=\int_{\partial D(r)}\log\frac{1}{\|\psi(x),a\|}d\pi_o^r(x).$$
Again,  set
$$\hat{N}_\psi(r,a)=\pi \sum_{x\in \psi^{-1}(a)\cap D(r)}g_r(o,x).$$
By the similar arguments as in the proof of Theorem \ref{first}, we  have
$$\hat{T}_\psi(r)=\hat{m}_\psi(r,a)+\hat{N}_\psi(r,a)+O(1).$$
 Note  that $m(r,\psi)=\hat{m}_\psi(r,\infty)+O(1),$ which yields that
 $$
   T(r,\psi)=\hat{T}_\psi(r)+O(1), \ \ \ T\Big(r,\frac{1}{\psi-a}\Big)= T(r,\psi)+O(1).
$$
Hence, we arrive at 
\begin{equation}\label{relation}
  T(r,\psi)+O(1)=\hat{T}_\psi(r)\leq T_\psi(r)+O(1).
\end{equation}

In order to prove the main results stated in Introduction, the Logarithmic Derivative Lemma (LDL) below is  important.
\begin{theorem}[LDL]\label{ldl2} Let $\psi$ be a nonconstant meromorphic function on $S.$ Let $\mathfrak{X}$ be a nowhere-vanishing holomorphic
 vector field over $S.$ Then
\begin{eqnarray*}
m\Big(r,\frac{\mathfrak{X}^k(\psi)}{\psi}\Big)  &\leq& \frac{3k}{2}\log T(r,\psi)+O\Big(\log^+\log T(r,\psi)-\kappa(r)r^2
+\log^+\log r\Big) \big{\|}
\end{eqnarray*}
with $\mathfrak X^j=\mathfrak X\circ\mathfrak X^{j-1}$ and $\mathfrak X^0=id,$   where 
$\kappa$ is defined by $(\ref{kappa}).$ In particular, if $S$ is the Poincar\'e disc,  then 
\begin{eqnarray*}
m\Big(r,\frac{\mathfrak{X}^k(\psi)}{\psi}\Big)  &\leq& \frac{3k}{2}\log T(r,\psi)+O\Big(\log^+\log T(r,\psi)+r\Big) \big{\|}.
\end{eqnarray*}
\end{theorem}
\textbf{Remark.}  LDL  is still valid for a general open Riemann surface by lifting $S$ to the  universal covering, see  arguments in Section 4.2.

 On $\mathbb P^1(\mathbb C),$ we take a singular metric
$$\Phi=\frac{1}{|\zeta|^2(1+\log^2|\zeta|)}\frac{\sqrt{-1}}{4\pi^2}d\zeta\wedge d\overline \zeta.$$
A direct computation gives that
\begin{equation*}\label{}
\int_{\mathbb P^1(\mathbb C)}\Phi=1, \ \ \ 2\pi\psi^*\Phi=\frac{\|\nabla_S\psi\|^2}{|\psi|^2(1+\log^2|\psi|)}\alpha.
\end{equation*}
Set
\begin{equation*}\label{ffww}
  T_\psi(r,\Phi)=\frac{1}{2\pi}\int_{D(r)}g_r(o,x)\frac{\|\nabla_S\psi\|^2}{|\psi|^2(1+\log^2|\psi|)}(x)dV(x).
\end{equation*}
 Using Fubini's theorem,
\begin{eqnarray*}
T_\psi(r,\Phi)
&=&\int_{D(r)}g_r(o,x)\frac{\psi^*\Phi}{\alpha}dV(x)  \\
&=&\pi\int_{\zeta\in\mathbb P^1(\mathbb C)}\Phi\sum_{x\in \psi^{-1}(\zeta)\cap D(r)}g_r(o,x) \\
&=&\int_{\zeta\in\mathbb P^1(\mathbb C)}N\big(r,1/(\psi-\zeta)\big)\Phi
\leq T(r,\psi)+O(1).
\end{eqnarray*}
Then we get
\begin{equation}\label{get}
  T_\psi(r,\Phi)\leq T(r,\psi)+O(1).
\end{equation}
\begin{proposition}\label{999a} Assume that $\psi(x)\not\equiv0.$ Then 
\begin{eqnarray*}
  && \frac{1}{2}\mathbb E_o\left[\log^+\frac{\|\nabla_S\psi\|^2}{|\psi|^2(1+\log^2|\psi|)}(X_{\tau_r})\right] \\
  &\leq&\frac{1}{2}\log T(r,\psi)+O\big{(}\log^+\log T(r,\psi)+r\sqrt{-\kappa(r)}+\log^+\log r\big{)}  \big{\|},
\end{eqnarray*}
 where $\kappa$ is defined by $(\ref{kappa}).$
\end{proposition}
\begin{proof} By Jensen's inequality
\begin{eqnarray*}
   \mathbb E_o\left[\log^+\frac{\|\nabla_S\psi\|^2}{|\psi|^2(1+\log^2|\psi|)}(X_{\tau_r})\right]
   &\leq&  \mathbb E_o\left[\log\Big{(}1+\frac{\|\nabla_S\psi\|^2}{|\psi|^2(1+\log^2|\psi|)}(X_{\tau_r})\Big{)}\right] \nonumber \\
    &\leq& \log^+\mathbb E_o\left[\frac{\|\nabla_S\psi\|^2}{|\psi|^2(1+\log^2|\psi|)}(X_{\tau_r})\right]+O(1). \nonumber
\end{eqnarray*}
Combine Lemma \ref{cal} and co-area formula with  (\ref{get})
\begin{eqnarray*}
   && \log^+\mathbb E_o\left[\frac{\|\nabla_S\psi\|^2}{|\psi|^2(1+\log^2|\psi|)}(X_{\tau_r})\right]  \\
   &\leq& \log^+\mathbb E_o\left[\int_0^{\tau_r}\frac{\|\nabla_S\psi\|^2}{|\psi|^2(1+\log^2|\psi|)}(X_{t})dt\right]
   +\log \frac{F(\hat{k},\kappa,\delta)e^{r\sqrt{-\kappa(r)}}\log r}{2\pi C}
    \\
   &\leq& \log T_\psi(r,\Phi)+\log F(\hat k,\kappa,\delta)+r\sqrt{-\kappa(r)}+\log^+\log r+O(1) \\
    &\leq& \log T(r,\psi)+O\Big(\log^+\log^+\hat k(r)+r\sqrt{-\kappa(r)}+\log^+\log r\Big),
\end{eqnarray*}
where
$$\hat k(r)=\frac{\log r}{C}\mathbb E_o\left[\int_0^{\tau_{r}}\frac{\|\nabla_S\psi\|^2}{|\psi|^2(1+\log^2|\psi|)}(X_{t})dt\right].$$
Indeed,  note that
$$
\hat k(r)=\frac{2\pi \log r}{C}T_\psi(r,\Phi)\leq \frac{2\pi\log r}{C} T(r,\psi).
$$
Hence, we  have the desired inequality.
\end{proof}
We first give LDL for the first-order derivative:
\begin{theorem}[LDL]\label{ldl1} Let $\psi$ be a nonconstant meromorphic function on $S.$ Let $\mathfrak{X}$ be a nowhere-vanishing holomorphic vector field over $S.$ Then
\begin{eqnarray*}
m\Big(r,\frac{\mathfrak{X}(\psi)}{\psi}\Big)&\leq&\frac{3}{2}\log T(r,\psi)+O\Big{(}\log^+\log T(r,\psi)-\kappa(r)r^2+\log^+\log r\Big{)}  \big{\|},
\end{eqnarray*}
where $\kappa$ is defined by $(\ref{kappa}).$ In particular, if $S$ is the Poincar\'e disc,  then
\begin{eqnarray*}
m\Big(r,\frac{\mathfrak{X}(\psi)}{\psi}\Big)&\leq&\frac{3}{2}\log T(r,\psi)+O\Big{(}\log^+\log T(r,\psi)+r\Big{)}  \big{\|}.
\end{eqnarray*}
\end{theorem}
\begin{proof} Write $\mathfrak{X}=a\frac{\partial}{\partial z},$ then $\|\mathfrak{X}\|^2=g|a|^2.$  We have
\begin{eqnarray*}
 m\Big(r,\frac{\mathfrak{X}(\psi)}{\psi}\Big)
&=& \int_{\partial D(r)}\log^+\frac{|\mathfrak{X}(\psi)|}{|\psi|}(x)d\pi^r_o(x) \\
&\leq& \frac{1}{2}\int_{\partial D(r)}\log^+\frac{|\mathfrak{X}(\psi)|^2}{\|\mathfrak{X}\|^2|\psi|^2(1+\log^2|\psi|)}(x)d\pi^r_o(x) \\
&&+
\frac{1}{2}\int_{\partial D(r)}\log(1+\log^2|\psi(x)|)d\pi^r_o(x)+\frac{1}{2}\int_{\partial D(r)}\log^+\|\mathfrak{X}_x\|^2d\pi^r_o(x) \\
&:=& A+B+C.
\end{eqnarray*}
We handle $A,
B,C$ respectively.  For $A,$ it yields from Proposition \ref{999a} that
\begin{eqnarray*}
A&=& \frac{1}{2}\int_{\partial D(r)}\log^+\frac{|a|^2\left|\frac{\partial \psi}{\partial z}\right|^2}{g|a|^2|\psi|^2(1+\log^2|\psi|)}(x)d\pi^r_o(x) \\
&=& \frac{1}{2}\int_{\partial D(r)}\log^+\frac{\|\nabla_S\psi\|^2}{|\psi|^2(1+\log^2|\psi|)}(x)d\pi^r_o(x) \\
 &\leq&\frac{1}{2}\log T(r,\psi)+O\Big{(}\log^+\log T(r,\psi)+r\sqrt{-\kappa(r)}+\log^+\log r\Big{)}.
\end{eqnarray*}
For $B,$ the Jensen's inequality implies that
\begin{eqnarray*}
B &\leq&  \int_{\partial D(r)}\log\Big(1+\log^+|\psi(x)|+\log^+\frac{1}{|\psi(x)|}\Big)d\pi^r_o(x) \\
&\leq& \log\int_{\partial D(r)}\Big(1+\log^+|\psi(x)|+\log^+\frac{1}{|\psi(x)|}\Big)d\pi^r_o(x) \\
&\leq& \log T(r,\psi)+O(1).
\end{eqnarray*}
Finally, we estimate $C.$ By the condition,  $\|\mathfrak{X}\|>0.$ Since $S$ is non-positively curved and $a$ is holomorphic, then $\log\|\mathfrak{X}\|$ is subharmonic, i.e., $\Delta_S\log\|\mathfrak X\|\geq0.$
Clearly, we have
\begin{equation}\label{h1}
\Delta_S\log^+\|\mathfrak{X}\|\leq \Delta_S\log \|\mathfrak{X}\|
\end{equation}
for $x\in S$ satisfying $\|\mathfrak{X}_x\|\neq1.$ Notice that  
\begin{equation}\label{h2}
\log^+\|\mathfrak{X}_x\|=0
\end{equation}
for $x\in S$ satisfying $\|\mathfrak{X}_x\|\leq1.$ Note that Dynkin formula cannot be  directly applied to $\log^+\|\mathfrak{X}\|,$ but by virtue of (\ref{h1}) and (\ref{h2}), it is not hard to 
verify
\begin{eqnarray}\label{ok}
\ \ \ \ \ C &=& \frac{1}{2}\mathbb E_o\left[\log^+\|\mathfrak{X}(X_{\tau_r})\|^2\right] \\
&\leq&  \frac{1}{4}\mathbb E_o\left[\int_0^{\tau_r}\Delta_S\log\|\mathfrak{X}(X_t)\|^2dt \right]+O(1) \nonumber \\
 &=& \frac{1}{4}\mathbb E_o\left[\int_0^{\tau_r}\Delta_S\log g(X_t)dt \right]+
 \frac{1}{4}\mathbb E_o\left[\int_0^{\tau_r}\Delta_S\log|a(X_t)|^2dt \right]+O(1) \nonumber \\
 &=& -\mathbb E_o\left[\int_0^{\tau_r}K_S(X_t)dt \right]+O(1) \nonumber \\
 &\leq& -\kappa(r)\mathbb E_o\big[\tau_r\big]+O(1), \nonumber
\end{eqnarray}
where we use the fact $K_S=-(\Delta_S\log g)/4.$ Thus, we prove the first assertion  by  $\mathbb E_o\left[\tau_r\right]\leq 4r^2,$  due to Proposition \ref{yyyy} below. When $S=\mathbb D$ with Poincar\'e metric,  $\kappa(r)\equiv-1$ and $g=2/(1-|z|^2)^{2}.$ Take $\mathfrak X=\partial/\partial z,$ then
$C$ in (\ref{ok}) is   estimated  as follows
\begin{eqnarray*}
C &=& \frac{1}{2}\mathbb E_o\left[\log^+\|\mathfrak{X}(X_{\tau_r})\|^2\right] \\
&=&  \frac{1}{2}\int_{\partial D(r)}\log\frac{2}{(1-(e^r-1)^2/(e^r+1)^2)^2}\frac{d\theta}{2\pi} \\
&=& \log\frac{\sqrt2}{1-(e^r-1)^2/(e^r+1)^2} \\
&\leq& r+O(1).
\end{eqnarray*}
This implies that the second assertion holds.
\end{proof}
 \begin{proposition}\label{yyyy}  We have
$$\mathbb E_o\big[\tau_r\big]\leq 4r^2.$$
\end{proposition}
\begin{proof}  The argument follows essentially from  Atsuji \cite{atsuji}, but here we provide a simpler proof though a rougher estimate.
Let $X_t$ be the Brownian motion in $S$ started at  $o\not=o_1,$ where $o_1\in D(r).$  Let  $r_1(x)$ be the distance function of $x$ from $o_1.$
Apply It$\rm{\hat{o}}$ formula to $r_1(x)$
\begin{equation}\label{kiss}
  r_1(X_t)-r_1(X_0)=B_t-L_t+\frac{1}{2}\int_0^t\Delta_Sr_1(X_s)ds,
\end{equation}
here $B_t$ is the standard Brownian motion in $\mathbb R,$ and $L_t$ is a local time on cut locus of $o,$ an increasing process
which increases only at cut loci of $o.$ Since $S$ is simply connected and  non-positively  curved, then
$$\Delta_Sr_1(x)\geq\frac{1}{r_1(x)}, \ \ L_t\equiv0.$$
By (\ref{kiss}), we arrive at
$$r_1(X_t)\geq B_t+\frac{1}{2}\int_0^t\frac{ds}{r_1(X_s)}.$$
Let $t=\tau_r$ and take expectation on both sides of the above inequality, then it yields that 
$$\max_{x\in \partial D(r)} r_1(x)\geq \frac{\mathbb E_o[\tau_r]}{2\max_{x\in \partial D(r)} r_1(x)}.$$
Let $o'\rightarrow o,$ 
we are led to the conclusion.
\end{proof}
Finally, let us prove Theorem \ref{ldl2}:
\begin{proof} Note that
\begin{eqnarray*}
m\Big(r,\frac{\mathfrak{X}^k(\psi)}{\psi}\Big)&\leq& \sum_{j=1}^k m\Big(r,\frac{\mathfrak{X}^j(\psi)}{\mathfrak{X}^{j-1}(\psi)}\Big).
\end{eqnarray*}
We conclude the proof by using  Proposition \ref{hello1} below.
\end{proof}
\begin{proposition}\label{hello1} We have
\begin{eqnarray*}
m\Big(r,\frac{\mathfrak{X}^{k+1}(\psi)}{\mathfrak{X}^{k}(\psi)}\Big)
&\leq& \frac{3}{2}\log T(r,\psi)+O\Big(\log^+\log T(r,\psi)-\kappa(r)r^2+\log^+\log r\Big) \big{\|},
\end{eqnarray*}
 where $\kappa$ is defined by $(\ref{kappa}).$ In particular, if $S$ is the Poincar\'e disc,  then 
 \begin{eqnarray*}
m\Big(r,\frac{\mathfrak{X}^{k+1}(\psi)}{\mathfrak{X}^{k}(\psi)}\Big)
&\leq& \frac{3}{2}\log T(r,\psi)+O\Big(\log^+\log T(r,\psi)+r\Big) \big{\|}.
\end{eqnarray*}
\end{proposition}
\begin{proof}
For the first assertion, we  claim that
\begin{eqnarray}\label{hello}
\ \ T\big{(}r,\mathfrak{X}^k(\psi)\big{)}
&\leq& 2^kT(r,\psi)+O\Big(\log T(r,\psi)-\kappa(r)r^2+\log^+\log r\Big).
\end{eqnarray}
By virtue of Theorem \ref{ldl1}, when $k=1$
\begin{eqnarray*}
T(r,\mathfrak{X}(\psi))
&=& m(r,\mathfrak{X}(\psi))+N(r,\mathfrak{X}(\psi)) \\
&\leq& m(r,\psi)+2N(r,\psi)+m\Big(r,\frac{\mathfrak{X}(\psi)}{\psi}\Big) \\
&\leq& 2T(r,\psi)+m\Big(r,\frac{\mathfrak{X}(\psi)}{\psi}\Big) \\
&\leq& 2T(r,\psi)+O\Big(\log T(r,\psi)-\kappa(r)r^2+\log^+\log r\Big)
\end{eqnarray*}
holds for $r>1$ outside a set of finite Lebesgue measure.
Assuming now that the claim holds for $k\leq n-1.$ By induction, we only need to prove the claim in the case when $k=n.$ By this claim for $k=1$ proved above and Theorem \ref{ldl1} repeatedly, we conclude that
\begin{eqnarray*}
 T\big{(}r,\mathfrak{X}^n(\psi)\big{)}
&\leq& 2T\big{(}r,\mathfrak{X}^{n-1}(\psi)\big{)}+O\Big(\log T\big{(}r,\mathfrak{X}^{n-1}(\psi)\big{)}-\kappa(r)r^2+\log^+\log r\Big) \\
&\leq& 2^{n}T(r,\psi)+O\Big(\log T(r,\psi)-\kappa(r)r^2+\log^+\log r\Big) \\
&&
+O\Big(\log T\big{(}r,\mathfrak{X}^{n-1}(\psi)\big{)}-\kappa(r)r^2+\log^+\log r\Big) \\
&\leq& 2^nT(r,\psi)+O\Big(\log T(r,\psi)-\kappa(r)r^2+\log^+\log r\Big) \\
&&+O\left(\log T\big{(}r,\mathfrak{X}^{n-1}(\psi)\big{)}\right) \\
&& \cdots\cdots\cdots \\
&\leq& 2^nT(r,\psi)+O\Big(\log T(r,\psi)-\kappa(r)r^2+\log^+\log r\Big). \ \ \ \ \
\end{eqnarray*}
So, the claim (\ref{hello}) is proved.
 Employing Theorem \ref{ldl1} and (\ref{hello}) to get
\begin{eqnarray*}
&& m\left(r,\frac{\mathfrak{X}^{k+1}(\psi)}{\mathfrak{X}^{k}(\psi)}\right) \\
&\leq& \frac{3}{2}\log T\big{(}r,\mathfrak{X}^k(\psi)\big{)}+O\Big(\log^+\log T(r,\mathfrak{X}^k(\psi))-\kappa(r)r^2+\log^+\log r\Big) \\
&\leq& \frac{3}{2}\log T(r,\psi)+O\Big(\log^+\log T(r,\psi)-\kappa(r)r^2+\log^+\log r\Big). \ \ \ \ \
\end{eqnarray*}
This proves the first assertion, and the second assertion is proved similarly   by replacing $\kappa(r)r^2$ by $-r$ due to the second conclusion of Theorem \ref{ldl1}.
\end{proof}

\section{An extension of H. Cartan's  theory}

\subsection{Cartan-Nochka's approach}~

 Let $S$ be an open Riemann surface with  a nowhere-vanishing holomorphic vector field $\mathfrak X.$
Let $$f:S\rightarrow\mathbb P^n(\mathbb C)$$ be a holomorphic curve into complex projective space with the Fubini-Study form $\omega_{FS}.$
 Locally, we may write $f=[f_0:\cdots:f_n],$  a reduced representation, i.e., $f_0=w_0\circ f,\cdots$ are local holomorphic functions without common zeros,
  where $w=[w_0:\cdots:w_n]$ denotes homogenous coordinate system of $\mathbb P^n(\mathbb C).$ Set $\|f\|^2=|f_0|^2+\cdots+|f_n|^2.$
  Noting that $\Delta_S\log\|f\|^2$ is independent of the choices of representations of $f,$ hence it is globally  defined on $S.$
 The height function of $f$ is defined by
$$ T_f(r)= \pi\int_{D(r)}g_r(o,x)f^*\omega_{FS}=\frac{1}{4}\int_{D(r)}g_r(o,x)\Delta_S\log\|f(x)\|^2dV(x). 
$$ 
Given a hyperplane  $H$ of $\mathbb P^n(\mathbb C)$ with defining function
$\hat{H}(w)=h_0w_0+\cdots+h_nw_n.$ Set $\|\hat{H}\|^2=|h_0|^2+\cdots+|h_n|^2.$
The counting function of $f$ with respect to $H$ is defined by
\begin{eqnarray*}
N_f(r,H) &=& \pi\int_{D(r)}g_r(o,x)dd^c\big{[}\log|\hat{H}\circ f(x)|^2\big{]} \\
&=&\frac{1}{4}\int_{D(r)}g_r(o,x)\Delta_S\log|\hat{H}\circ f(x)|^2dV(x).
\end{eqnarray*}
We define the proximity function of $f$ with respect to $H$  by
$$m_f(r,H)=\int_{\partial D(r)}\log\frac{\|\hat{H}\|\|f(x)\|}{|\hat{H}\circ f(x)|}d\pi_o^r(x).$$
\begin{proposition}\label{v000} Assume that $f_k\not\equiv0$ for some $k.$ We have
$$\max_{0\leq j\leq n}T\Big(r,\frac{f_j}{f_k}\Big)\leq T_f(r)+O(1).$$
\end{proposition}
\begin{proof} Clearly,  $f_j/f_k$ is well defined on $S.$  From (\ref{relation}), we get
\begin{eqnarray*}
 T\Big(r,\frac{f_j}{f_k}\Big)
 &\leq&T_{f_j/f_k}(r)+O(1)     \\
 &\leq& \frac{1}{4}\int_{D(r)}g_r(o,x)\Delta_S\log\Big{(}\sum_{j=0}^n|f_j(x)|^2\Big{)}dV(x)+O(1) \\
  &=& T_{f}(r)+O(1).
\end{eqnarray*}
This completes the proof.
\end{proof}
\noindent\textbf{Wronskian determinants.}
 Let
 $H_1,\cdots, H_q$ be $q$ hyperplanes of $\mathbb P^n(\mathbb C)$ in $N$-subgeneral position with defining functions given by
$$\hat{H}_j(w)=\sum_{k=0}^nh_{jk}w_{k}, \ \ 1\leq j\leq q.$$
 Assume that $f$ is linearly non-degenerate.  We  define Wronskian determinant and logarithmic
Wronskian determinant of $f$ with respect to $\mathfrak{X}$ respectively  by
$$W_\mathfrak{X}(f_0,\cdots,f_n)=\left|
  \begin{array}{ccc}
   f_0 & \cdots & f_n \\
   \mathfrak{X}(f_0) & \cdots & \mathfrak{X}(f_n) \\
    \vdots & \vdots & \vdots \\
     \mathfrak{X}^n(f_0) & \cdots & \mathfrak{X}^n(f_n) \\
  \end{array}
\right|, \ \ \Delta_\mathfrak{X}(f_0,\cdots,f_n)=\left|
  \begin{array}{ccc}
   1 & \cdots & 1 \\
   \frac{\mathfrak{X}(f_0)}{f_0} & \cdots & \frac{\mathfrak{X}(f_n)}{f_n} \\
    \vdots & \vdots & \vdots \\
     \frac{\mathfrak{X}^n(f_0)}{f_0} & \cdots & \frac{\mathfrak{X}^n(f_n)}{f_n} \\
  \end{array}
\right|.$$
For a $(n+1)\times(n+1)$-matrix $A$  and a nonzero meromorphic function $\phi$  on $S,$ we can check the following basic properties:
\begin{eqnarray*}
\Delta_\mathfrak{X}(\phi f_0,\cdots,\phi f_n)&=&\Delta_\mathfrak{X}(f_0,\cdots,f_n), \\
W_\mathfrak{X}(\phi f_0,\cdots,\phi f_n)&=&\phi^{n+1}W_\mathfrak{X}(f_0,\cdots,f_n), \\
W_\mathfrak{X}\big{(}(f_0,\cdots,f_n)A\big{)}&=&\det(A)W_\mathfrak{X}(f_0,\cdots,f_n), \\
W_\mathfrak{X}(f_0,\cdots,f_n)&=&\Big{(}\prod_{j=0}^nf_j\Big{)}\Delta_\mathfrak{X}(f_0,\cdots,f_n).
\end{eqnarray*}
Clearly,  $\Delta_\mathfrak{X}(f_0,\cdots,f_n)$ is globally defined on $S.$

\noindent\textbf{Nochka weights.}
For a subset $Q=\{i_1,\cdots,i_k\}\subseteq\{1,\cdots,q\},$ we define
$${\rm{rank}}(Q)={\rm{rank}}\left(
  \begin{array}{ccc}
    h_{i_10} & \cdots & h_{i_1n} \\
    \vdots & \vdots & \vdots \\
   h_{i_k0} & \cdots & h_{i_kn} \\
  \end{array}
\right).$$
 \begin{lemma}[\cite{nochka}]\label{nochka} Let $H_1,\cdots,H_q$ be hyperplanes of $\mathbb P^n(\mathbb C)$ in $N$-subgeneral position with $q>2N-n+1.$
   Then there exists rational constants $\gamma_1,\cdots,\gamma_q$ satisfying the following conditions$:$

 $\rm{(i)}$ $0<\gamma_j\leq1$ for $1\leq j\leq q;$

 $\rm{(ii)}$ Set $\gamma=\max_{1\leq j\leq q}\gamma_j,$ we have
 $$\frac{n+1}{2N-n+1}\leq\gamma\leq\frac{n}{N}, \ \ \
  \gamma(q-2N+n-1)=\sum_{j=1}^q\gamma_j-n-1.$$

 $\rm{(iii)}$ If $Q\subseteq\{1,\cdots,q\}$ with $0<|Q|\leq N+1,$ then $\sum_{j\in Q}\gamma_j\leq {\rm{rank}}(Q).$
 
 Here, $\gamma_1,\cdots,\gamma_q$ are called  the Nochka weights and $\gamma$ is called  the Nochka constant.
 \end{lemma}

  \begin{lemma}[\cite{nochka}]\label{nochka1} Let $H_1,\cdots,H_q$ be hyperplanes of $\mathbb P^n(\mathbb C)$ in $N$-subgeneral position with $q>2N-n+1.$
Let $\gamma_1,\cdots,\gamma_q$ be Nochka weights for $H_1,\cdots,H_q$ and let $\beta_1,\cdots,\beta_q$ be arbitrary constants not less than $1.$
Then for each subset  $Q\subseteq\{1,\cdots,q\}$ with $0<|Q|\leq N+1,$ there are distinct $j_1,\cdots,j_{{\rm{rank}}(Q)}\in Q$ such that
$${\rm{rank}}\left(\{j_1,\cdots,j_{{\rm{rank}}(Q)}\}\right)={\rm{rank}}(Q),\ \ \ \prod_{j\in Q}\beta_j^{\gamma_j}\leq\prod_{i=1}^{{\rm{rank}}(Q)}\beta_{j_i}.$$
 \end{lemma}
We need the following preliminary results:
\begin{proposition}\label{nochka2} Let $H_1,\cdots,H_q$ be hyperplanes of $\mathbb P^n(\mathbb C)$ in $N$-subgeneral position with $q>2N-n+1.$
Let $\gamma$ and $\gamma_1,\cdots,\gamma_q$ be  Nochka constant and weights for $H_1,\cdots,H_q.$ Then there exists a constant $C>0$ determined by
 $\hat{H}_1\circ f,\cdots,\hat{H}_q\circ f$ such that
 $$\|f\|^{\gamma(q-2N+n-1)}\leq \frac{C\prod_{j=1}^q|\hat{H}_j\circ f|^{\gamma_j}}{|W_\mathfrak{X}(f_0,\cdots,f_n)|}\sum_{Q\subseteq\{1,\cdots,q\},\ |Q|=n+1}
 \left|\Delta_\mathfrak{X}\big{(}\hat{H}_j\circ f, j\in Q\big{)}\right|.$$
 \end{proposition}
 \begin{proof} Note from the definition of $N$-subgeneral position that for each point $w\in\mathbb P^n(\mathbb C),$ there is a subset $Q\subseteq\{1,\cdots,q\}$ with $|Q|=q-N-1$ such that
 $\prod_{j\in Q}\hat{H}_j(w)\neq0.$ Hence, there is a constant $C_1>0$ such that
 $$C_1^{-1}<\sum_{|Q|=q-N-1}\prod_{j\in Q}\left(\frac{|\hat{H}_j(w)|}{\|\hat{H}_j\|\|w\|}\right)^{\gamma_j}<C_1, \ \ ^\forall w\in\mathbb P^n(\mathbb C).$$
 Set $R=\{1,\cdots,q\}\setminus Q$ and rewrite
  $$\prod_{j\in Q}\left(\frac{|\hat{H}_j(w)|}{\|\hat{H}_j\|\|w\|}\right)^{\gamma_j}=
  \prod_{j\in R}\left(\frac{\|\hat{H}_j\|\|w\|}{|\hat{H}_j(w)|}\right)^{\gamma_j}\cdot
  \frac{\prod_{j=1}^q|\hat{H}_j(w)|^{\gamma_j}}{\prod_{j=1}^q\left(\|\hat{H}_j\|\|w\|\right)^{\gamma_j}}.$$
 Since ${\rm{rank}}(R)=n+1,$  Lemma \ref{nochka} (ii) and Lemma \ref{nochka1} implies that there is a subset $R'\subseteq R$ with $|R'|=n+1$ such  that 
  \begin{eqnarray}\label{bzd}
 \prod_{j\in Q}\left(\frac{|\hat{H}_j(w)|}{\|\hat{H}_j\|\|w\|}\right)^{\gamma_j}&\leq&
 \prod_{j\in R'}\frac{\|\hat{H}_j\|\|w\|}{|\hat{H}_j(w)|}
 \cdot\frac{\prod_{j=1}^q|\hat{H}_j(w)|^{\gamma_j}}{\prod_{j=1}^q\|\hat{H}_j\|\|w\|^{\gamma(q-2N+n-1)+n+1}}  \nonumber \\
 &=& \frac{1}{\prod_{j\in R'}|\hat{H}_j(w)|}
 \cdot\frac{\prod_{j=1}^q|\hat{H}_j(w)|^{\gamma_j}}{\prod_{j=1}^q\|\hat{H}_j\|\cdot \|w\|^{\gamma(q-2N+n-1)}}. 
 \end{eqnarray}
 By the property of Wronskian determinant given before, we see that
 $$c(R'):=\frac{|W_\mathfrak{X}(f_0,\cdots,f_n)|}{|W_\mathfrak{X}\big{(}\widehat{H}_f\circ f, j\in R'\big{)}|}$$
 is a positive number depending on $R'$. Hence, it yields from (\ref{bzd}) that
 \begin{eqnarray*}
 &&\prod_{j\in Q}\left(\frac{|\hat{H}_j\circ f|}{\|\widehat{H}_j\|\|f\|}\right)^{\gamma_j} \\
 &\leq&\frac{c(R')}{\|f\|^{\gamma(q-2N+n-1)}\prod_{j=1}^q\|\hat{H}_j\|}\cdot
 \frac{\prod_{j=1}^q|\hat{H}_j\circ f|^{\gamma_j}}{|W_\mathfrak{X}(f_0,\cdots,f_n)|}
 \cdot\frac{|W_\mathfrak{X}\big{(}\hat{H}_f\circ f, j\in R'\big{)}|}{\prod_{j\in R'}|\hat{H}_j\circ f|}  \\
 &=&\frac{c(R')}{\|f\|^{\gamma(q-2N+n-1)}\prod_{j=1}^q\|\hat{H}_j\|}\cdot
 \frac{\prod_{j=1}^q|\hat{H}_j\circ f|^{\gamma_j}}{|W_\mathfrak{X}(f_0,\cdots,f_n)|}
 \cdot \left|\Delta_\mathfrak{X}\big{(}\hat{H}_f\circ f, j\in R'\big{)}\right|,
 \end{eqnarray*}
 where we use the relation between Wronskian determinant and logarithmic Wronskian determinant stated before. The proposition is proved by setting  $C=C_1\max_{R'}\{c(R')\}/\prod_{j=1}^q\|\hat{H}_j\|.$
 \end{proof}
 \begin{proposition}\label{nochka3} Assume the same notations as in Proposition $\ref{nochka2}.$ Then
   $$\sum_{j=1}^q\gamma_j(\hat{H}_j\circ f)-\big{(}W_\mathfrak{X}(f_0,\cdots,f_n)\big{)}\leq
   \sum_{j=1}^q\gamma_j\sum_{a\in S}\min\big{\{}{\rm{ord}}_a\hat{H}_j\circ f, n\big{\}}\cdot a$$
holds  as a divisor on  $S$ with rational coefficients.
 \end{proposition}
 \begin{proof}  The proof essentially follows Fujimoto \cite{Fujimoto}.
 We obverse that
 $${\rm{ord}}_a\hat{H}_j\circ f=\min\big{\{}{\rm{ord}}_a\hat{H}_j\circ f,n\big{\}}+
 \big{(}{\rm{ord}}_a\hat{H}_j\circ f-n\big{)}^+.$$
 Hence, the inequality claimed in the proposition is equivalent to
 \begin{equation}\label{claim}
   \sum_{j=1}^q\gamma_j\sum_{a\in S}\big{(}{\rm{ord}}_a\hat{H}_j\circ f-n\big{)}^+\cdot a\leq\big{(}W_\mathfrak{X}(f_0,\cdots,f_n)\big{)}.
 \end{equation}
 In what follows, we show (\ref{claim}) holds. Take an arbitrary point $a\in S$ and put
 $$E=\big{\{}j\in\{1,\cdots,n\}: {\rm{ord}}_a\hat{H}_j\circ f\geq n+1\big{\}}.$$
 The assumption of $N$-subgeneral position implies $|E|\leq N.$ We may assume that $E\neq\emptyset.$ Let 
 $m_1>\cdots>m_t\geq n+1$ be the orders ${\rm ord}_a\hat{H}_j\circ f, \ j \in E$ in order from the largest to the smallest.  Take a sequence of subsets of $E$
 $$\emptyset=E_0\neq E_1\subseteq E_2\subseteq\cdots\subseteq E_t=E$$
 so that ${\rm{ord}}_a\hat{H}_j\circ f=m_k$ for all $j\in E_k\setminus E_{k-1}.$ For each $E_k,$ we take a subset $F_k\subseteq E_k$ such that
 $|F_k|={\rm{rank}}(F_k)={\rm{rank}}(E_{k})$ and $F_{k-1}\subseteq F_k.$ Whence, it yields that $|F_k\setminus F_{k-1}|={\rm{rank}}(E_k)-{\rm{rank}}(E_{k-1}).$
 Set  $m_k'=m_k-n.$ Then
$$\sum_{j=1}^q\gamma_j\big{(}{\rm{ord}}_a\hat{H}_j\circ f-n\big{)}^+ = \sum_{j\in E}\gamma_j\big{(}{\rm{ord}}_a\hat{H}_j\circ f-n\big{)}= \sum_{k=1}^t\sum_{j\in E_k\setminus E_{k-1}}\gamma_jm'_k.$$
For the last term,  by (iii) in Lemma \ref{nochka} we get 
\begin{eqnarray*}
\sum_{k=1}^t\sum_{j \in E_k\setminus E_{k-1}}\gamma_jm'_k
&=& (m'_1-m'_2)\sum_{j\in E_1}\gamma_j+\cdots+m'_t\sum_{j\in E_t}\gamma_j \\
&\leq& (m'_1-m'_2){\rm{rank}}(E_{1})+\cdots+m'_t{\rm{rank}}(E_{t}) \\
&=& {\rm{rank}}(E_1)m'_1+\cdots+
 \big{(}{\rm{rank}}(E_t)-{\rm{rank}}(E_{t-1})\big{)}m'_t \\
 &=& |F_1|m_1'+|F_2\setminus F_1|m'_2+\cdots+|F_t\setminus F_{t-1}|m'_t.
\end{eqnarray*}
Hence, 
$$\sum_{j=1}^q\gamma_j\big{(}{\rm{ord}}_a\hat{H}_j\circ f-n\big{)}^+\leq|F_1|m_1'+|F_2\setminus F_1|m'_2+\cdots+|F_t\setminus F_{t-1}|m'_t.$$
On the other hand, since $\mathfrak X$ is  vanishing nowhere,
the order of $W_\mathfrak{X}(f_0,\cdots,f_n)$ at $a$ does not change by a linear reversible transformation of $f_0,\cdots,f_n.$ Put
$F_t=\{j_0,\cdots,j_k\},$ we  may assume that  $f_0=\hat{H}_{j_0}\circ f,\cdots, f_k=\hat{H}_{j_k}\circ f$ without loss of generality.  
A simple computation for Wronskian determinant gives
$${\rm{ord}}_aW_\mathfrak{X}(f_0,\cdots,f_n)\geq |F_1|m_1'+|F_2\setminus F_1|m'_2+\cdots+|F_t\setminus F_{t-1}|m'_t.$$
Therefore, (\ref{claim}) is confirmed.
 \end{proof}

\subsection{Proof of Theorem \ref{thm}}~

 Let  $\pi:\tilde{S}\rightarrow S$ be the  (analytic) universal covering. By the pull-back of $\pi,$ $\tilde{S}$ can be equipped with the induced metric
 from the
metric of $S.$ In such case, $\tilde{S}$ is a simply-connected and complete open Riemann surface of non-positive Gauss curvature.
Take a diffusion process $\tilde{X}_t$ in $\tilde{S}$  so that $X_t=\pi(\tilde{X}_t),$ then
$\tilde{X}_t$ becomes a Brownian motion with generator  $\Delta_{\tilde{S}}/2$ which is induced from the pull-back metric.
Let $\tilde{X}_t$ start from $\tilde{o}\in\tilde{S}$  with $o=\pi(\tilde{o}),$ then we have
$$\mathbb E_o[\phi(X_t)]=\mathbb E_{\tilde{o}}\big{[}\phi\circ\pi(\tilde{X}_t)\big{]}$$
for $\phi\in \mathscr{C}_{\flat}(S).$ Set $$\tilde{\tau}_r=\inf\big{\{}t>0: \tilde{X}_t\not\in \tilde{D}(r)\big{\}},$$ where
$\tilde{D}(r)$ is a geodesic disc centered at $\tilde{o}$ with radius $r$ in $\tilde{S}.$
 If necessary, one can extend the filtration in probability space where $(X_t,\mathbb P_o)$ are defined so that $\tilde{\tau}_r$ is a stopping time with
 respect to a filtration where the stochastic calculus of $X_t$ works.
By the above arguments, we could assume $S$ is simply connected by lifting $f$ to the covering.
 \begin{lemma}\label{ldl3} Let $Q\subseteq\{1,\cdots,q\}$ with $|Q|=n+1.$ If $S$ is simply connected, then we have
 \begin{eqnarray*}
 m\Big(r,\Delta_\mathfrak{X}\big{(}\hat{H}_k\circ f, k\in Q\big{)}\Big)
 &\leq& O\Big(\log T_f(r)-\kappa(r)r^2+\log^+\log r\Big)  \big{\|},
 \end{eqnarray*}
where $\kappa$ is defined by $(\ref{kappa}).$ In particular, if $S$ is the Poincar\'e disc,  then 
 \begin{eqnarray*}
 m\Big(r,\Delta_\mathfrak{X}\big{(}\hat{H}_k\circ f, k\in Q\big{)}\Big)
 &\leq& O\Big(\log T_f(r)+r\Big)  \big{\|}.
 \end{eqnarray*}
 \end{lemma}
 \begin{proof} We write $Q=\{j_0,\cdots,j_n\}$ and suppose that $\hat{H}_{j_0}\circ f\not\equiv0$ without loss of generality.
 The property of logarithmic Wronskian determinant indicates
 $$\Delta_\mathfrak{X}\big{(}\hat{H}_{j_0}\circ f,\cdots, \hat{H}_{j_n}\circ f\big{)}=
 \Delta_\mathfrak{X}\left(1,\frac{\hat{H}_{j_1}\circ f}{\hat{H}_{j_0}\circ f}, \cdots, \frac{\hat{H}_{j_n}\circ f}{\hat{H}_{j_0}\circ f}\right).$$
 Since $\hat{H}_{j_0}\circ f,\cdots,\hat{H}_{j_n}\circ f$ are linear forms of $f_0,\cdots, f_n,$ by Theorem \ref{ldl2} and
 Proposition
 \ref{v000} we have
 $$ m\Big(r,\Delta_\mathfrak{X}\big{(}\hat{H}_k\circ f, k\in Q\big{)}\Big) \\
  \leq O\Big(\log T_f(r)-\kappa(r)r^2+\log^+\log r\Big).$$
 We prove the first assertion, and the second assertion is proved similarly   by replacing $\kappa(r)r^2$ by $-r.$
 \end{proof}

We now prove Theorem \ref{thm}:
 \begin{proof} By Proposition \ref{nochka2} and Dynkin formula
 \begin{eqnarray*}
 &&\gamma(q-2N+n-1)dd^c\log\|f\|^2
 -\sum_{j=1}^q\gamma_j\big{(}\hat{H}_j\circ f\big{)}+\big{(}W_\mathfrak{X}(f_0,\cdots,f_n)\big{)} \\
 &\leq&
2dd^c
 \Big{[}\log\sum_{Q\subseteq\{1,\cdots,q\},\ |Q|=n+1}
 \left|\Delta_\mathfrak{X}\big{(}\hat{H}_k\circ f, k\in Q\big{)}\right|\Big{]}
 \end{eqnarray*}
 in the sense of currents.  Integrating both sides of the above inequality, then it yields from Proposition \ref{nochka3} that 
 \begin{eqnarray*}
 &&(q-2N+n-1)T_f(r)- \sum_{j=1}^qN^{[n]}_f(r,H_j) \\
 &\leq&
 \frac{2\pi}{\gamma}\int_{D(r)}g_r(o,x)dd^c
 \Big{[}\log\sum_{Q\subseteq\{1,\cdots,q\},\ |Q|=n+1}
 \left|\Delta_\mathfrak{X}\big{(}\hat{H}_k\circ f, k\in Q\big{)}\right|\Big{]} \\
 &=& \frac{1}{2\gamma}\int_{D(r)}g_r(o,x)\Delta_S\log\sum_{Q\subseteq\{1,\cdots,q\},\ |Q|=n+1}
 \left|\Delta_\mathfrak{X}\big{(}\hat{H}_k\circ f, k\in Q\big{)}\right|dV(x).
 \end{eqnarray*}
 Applying co-area formula and Dynkin formula to the last term, we get
 \begin{eqnarray*}
 && \frac{1}{2}\int_{D(r)}g_r(o,x)\Delta_S\log\sum_{Q\subseteq\{1,\cdots,q\},\ |Q|=n+1}
 \left|\Delta_\mathfrak{X}\big{(}\hat{H}_k\circ f, k\in Q\big{)}\right|dV(x) \\
 &=& \int_{\partial D(r)}\log\sum_{Q\subseteq\{1,\cdots,q\},\ |Q|=n+1}
 \left|\Delta_\mathfrak{X}\big{(}\hat{H}_k\circ f, k\in Q\big{)}\right|d\pi_o^r(x)+O(1) \\
 &\leq& \sum_{Q\subseteq\{1,\cdots,q\},\ |Q|=n+1}m\left(r,\Delta_\mathfrak{X}\big{(}\hat{H}_k\circ f, k\in Q\big{)}\right)+O(1) \\
 &\leq& O\Big(\log T(r,\psi)-\kappa(r)r^2+\log^+\log r\Big).
 \end{eqnarray*}
 The last step follows from Lemma \ref{ldl3}. This proves  the  theorem. 
 \end{proof}
 
Corollary \ref{ccc1} can be  proved similarly  just
 by replacing $\kappa(r)r^2$ by $-r,$ due to the second conclusion of Lemma \ref{ldl3}. 
 
  Let $H$ be a hyperplane of $\mathbb P^n(\mathbb C)$ such that $H\not\supseteq f(S).$  The \emph{$k$-defect} $\delta^{[k]}_f(H)$ of $f$
 at $k$-level with respect to $H$ is defined by
$$\delta^{[k]}_f(H)=1-\limsup_{r\rightarrow\infty}\frac{N^{[k]}_f(r,H)}{T_f(r)},$$
where $N^{[k]}_f(r,H)$ is the  $k$-truncated counting function. By  definition, we can derive Corollary \ref{defect1}  immediately.

\section{Vanishing theorem}
\subsection{LDL for logarithmic jet differentials}~

Let $S$ be an open Riemann surface with a nowhere-vanishing holomorphic vector field $\mathfrak{X}.$   Let $M$ be  a complex   manifold of complex dimension $n.$ First, we  introduce jet bundles.

Fix a point $x_0\in S.$  Now consider a holomorphic curve $f:U_1\rightarrow M$ defined in an open neighborhood $U_1$ of $x_0$ with $f(x_0)=y,$ and
 another holomorphic curve $g:U_2\rightarrow M$ defined in an open neighborhood $U_2$ of $x_0$ with $g(x_0)=y,$
we define an equivalent relation $f\thicksim g$
if there exists an open neighborhood
 $U_3\subseteq U_1\cap U_2$ of $x_0$ such that $f= g$ on $U_3.$
Denote by  ${\rm{Hol}}_{loc}((S,x_0),(M,y))$ the set of  such equivalent classes.
 Locally, we  write  $f=(f_1,\cdots,f_n),$  where $f_j=\zeta_j\circ f$  near $x_0,$ here $(\zeta_1,\cdots,\zeta_n)$ is a local holomorphic coordinate system near $y.$ 
 Put
$$\mathfrak{X}^j(f)=\big{(}\cdots,\mathfrak{X}^j(f_i),\cdots\big{)}.$$
For $f,g\in{\rm{Hol}}_{loc}((S,x_0),(M,y)),$ we speak of $f\thicksim^{k}g$ with respect to $\mathfrak{X}$ if 
 $$\mathfrak{X}^jf(x_0)=\mathfrak{X}^jg(x_0), \ \ 1\leq j\leq k.$$
This relation is independent of the choices of  local holomorphic coordinates,  and defines  an equivalent relation.
Let  $j_{k,\mathfrak{X}}f$  be this equivalent class. Set
\begin{eqnarray*}
J_k(M,\mathfrak{X})_y&=&\Big{\{}j_{k,\mathfrak{X}}f: \ f\in{\rm{Hol}}_{loc}\big((S,x_0),(M,y)\big)\Big{\}},\\
  J_k(M,\mathfrak{X})&=&\bigsqcup_{y\in M} J_k(M,\mathfrak{X})_y.
\end{eqnarray*}
Apparently,  $J_k(M,\mathfrak{X})_y\cong\mathbb C^{nk},$ i.e., each  $v\in J_k(M,\mathfrak{X})_y$  is represented by
$$\Big{(}\mathfrak{X}^j(\zeta_i\circ\phi)(x_0): \ 1\leq i\leq n, \ 1\leq j\leq k\Big{)}$$
for some $\phi\in {\rm{Hol}}_{loc}((S,x_0),(M,y)).$ 
Write $\mathfrak X=a\frac{\partial}{\partial z}$ locally,  then $\mathfrak X$ induces a global holomorphic 1-form $\eta:=dz/a$ without zeros on $S.$  For every $x\in S,$ there exists an open neighborhood $U(x)$ of $x$ such that 
\begin{equation}\label{inte}
 \int_\gamma \eta=0
 \end{equation}
  for any  closed Jordan curve $\gamma$ in $U(x).$ In fact, 
it is known \cite{GR},  there exists a nowhere-vanishing holomorphic 1-form $\eta$ on $S$ such that
the contour integral (\ref{inte})  equals  0 along any closed Jordan curve $\gamma$ in $S.$
 We can choose $\mathfrak X$ induced from $\eta.$
For $x\in S,$ define
$$\widehat{x-x_0}:=\int_{x_0}^x\eta,$$
which is independent of the paths of integration, and hence is a holomorphic function of $x.$ 
 If $f$ is defined in an open neighborhood of another point $x_1\in S,$ 
then we  consider the  power series of $\widehat{x-x_0}$ that 
$$f_{x_1}(x):=f(x_1)+\mathfrak{X}f(x_1)\cdot\widehat{x-x_0}+\frac{\mathfrak{X}^2f(x_1)}{2!}
\cdot\widehat{x-x_0}^2+\cdots$$
which  is convergent in an open  neighborhood  of $x_0,$ and  holomorphic on this neighborhood.
It is trivial to check  that 
$$\mathfrak X^k f_{x_1}(x_0)=\mathfrak X^kf(x_1), \ \ k=0,1,\cdots$$
Hence,  $f$ gives a $k$-jet $j_{k,\mathfrak X}f_{x_1}(x_0)$ in $J_k(M,\mathfrak{X})_{f(x_0)}$ at $x_1.$ In further,  $f$ induces naturally 
$$J_{k,\mathfrak{X}}f:  x\mapsto j_{k,\mathfrak{X}}f_x(x_0)\in J_k(M,\mathfrak{X})_{f(x)}$$
for $x$ in some open neighborhood of  $x_1,$ which is called the \emph{$k$-jet lift} of $f.$

For every  holomorphic curve $f:S\to M,$    $f$ defines   a $k$-jet in $J_k(M,\mathfrak{X})_{f(x)}$ at every  $x\in S$ in a natural way,  i.e.,
$$J_{k,\mathfrak{X}}f(x)=\Big{(}\mathfrak{X}^j(\zeta_i\circ f)(x): \ 1\leq i\leq n, \ 1\leq j\leq k\Big{)}, \ \  ^\forall x\in S.$$
 \ \ \ \  We equip a complex structure to make $\pi_k:J_k(M,\mathfrak{X})\rightarrow M$ be a $\mathbb{C}^{nk}$-fiber bundle over $M,$
  where $\pi_k$ is the natural projection, and so $J_k(M,\mathfrak{X})$ becomes a complex manifold. 
 Note that  $J_1(M,\mathfrak{X})\cong T_M$ (holomorphic tangent bundle over $M$), but, in general,    $J_k(M,\mathfrak{X})$ is not a vector bundle for $k>1$.

A (holomorphic) jet differential $\omega$ of order $k$ and weighted degree $m$  on  an open set of $M$ (with  a local holomorphic coordinate  system $\zeta=(\zeta_1,\cdots,\zeta_n)$)
is   a homogeneous polynomial 
in $d^i\zeta_j$ $(1\leq i\leq k, 1\leq j\leq n)$
 of the form
$$\omega=\sum_{|l_1|+\cdots+k|l_k|=m}a_{l_1\cdots l_k}(\zeta)d\zeta^{l_1}\cdots d^k\zeta^{l_k}$$
with holomorphic function coefficients $a_{l_1\cdots l_k}(\zeta),$  
which is also simply  called   a $k$-jet differential of degree $m.$
We use  $E_{k,m}^{GG}T^*_M$ to  denote the sheaf of germs of $k$-jet differentials of degree $m.$

Let $D$ be a reduced divisor on $M.$ Then 
a logarithmic $k$-jet differential $\omega$ of degree $m$ along $D$ is  a $k$-jet differential of degree $m$ with  possible logarithmic poles along $D.$
Namely, along  $D,$ $\omega$ is locally  a homogeneous polynomial  in
$$d^s\log \sigma_1, \cdots,d^s\log\sigma_r,d^s\sigma_{r+1},\cdots,d^s\sigma_n, \ \ 1\leq s\leq k$$
of weighted degree $m,$ where $\sigma_1,\cdots,\sigma_r$ are irreducible, and $\sigma_1\cdots\sigma_r=0$ is a local defining equation of $D.$
 Let $E_{k,m}^{GG}T^*_M(\log D)$  denote the sheaf of germs of logarithmic $k$-jet differentials of degree $m$ over $M.$
 We note  that  $\omega(J_{k,\mathfrak{X}}f)$ is a meromorphic function on $S$ for $\omega\in H^0(M, E_{k,m}^{GG}T^*_M(\log D)).$ 

Suppose that  $M$ is a compact complex manifold with an 
ample divisor $A$ on $M.$  
In what follows we  extend Theorem \ref{ldl2} to logarithmic jet differentials.

\begin{theorem}[LDL for logarithmic  jet differentials]\label{xxx} Let $f:S\rightarrow M$ be a holomorphic curve
such that $f(S)\not\subseteq D,$ where
 $D$ is a reduced divisor  on $M.$   Then
$$m\big{(}r,\omega(J_{k,\mathfrak{X}}f)\big{)}\leq O\Big{(}\log T_{f,A}(r)-\kappa(r)r^2+\log^+\log r\Big{)} \big{\|}$$
for  $\omega\in H^0(M, E_{k,m}^{GG}T^*_M(\log D)),$ where $\kappa$ is defined by $(\ref{kappa}).$ In particular, if $S$ is the Poincar\'e disc,  then 
$$m\big{(}r,\omega(J_{k,\mathfrak{X}}f)\big{)}\leq O\Big{(}\log T_{f,A}(r)+r\Big{)} \big{\|}$$
for  $\omega\in H^0(M, E_{k,m}^{GG}T^*_M(\log D)).$
\end{theorem}
\begin{proof}
We  can regard $M$ as an algebraic submanifold  of a complex projective space.  By  Hironaka's desingularization  \cite{Hir},  there exists a
brow-up
$\tau:\tilde{M}\rightarrow M$
with center at  singular locus of $D$  
such that $f^{-1}(D)$ is of  simple  normal crossing type. 
Let a holomorphic curve $\tilde{f}:S\rightarrow\tilde{M}$ such that  
  $\tau\circ \tilde{f}=f.$ Since  $\omega(J_{k,\mathfrak{X}}f)=\tau^*\omega(J_{k,\mathfrak{X}}\tilde{f}),$
   it  suffices to certify this  assertion for $\tau^*\omega(J_{k,\mathfrak{X}}\tilde{f})$ on $\tilde{M}.$
   So, we  assume that $D$ has only simple normal crossings.
By  Ru-Wong's arguments (\cite{ru}, Page 231-233),
there exists a finite  open covering $\{U_\lambda\}$ of $M$ and  rational functions
$w_{\lambda1},\cdots,w_{\lambda n}$ on $M$ for each $\lambda,$ such that $w_{\lambda1},\cdots,w_{\lambda n}$ are holomorphic on $U_\lambda$ as well as
\begin{eqnarray*}
  dw_{\lambda1}\wedge\cdots\wedge dw_{\lambda n}(x)\neq0, & & \ ^\forall x\in U_{\lambda}, \\
  D\cap U_{\lambda}=\big{\{}w_{\lambda1}\cdots w_{\lambda \alpha_\lambda}=0\big{\}}, && \ ^\exists \alpha_{\lambda}\leq n.
\end{eqnarray*}
Hence, for each $\lambda$ we get
$$\omega|_{U_{\lambda}}=P_\lambda\Big(\frac{d^iw_{\lambda j}}{w_{\lambda j}}, \ d^hw_{\lambda l}\Big),
\ \  1\leq i,h\leq k, \ 1\leq j\leq \alpha_\lambda, \ \alpha_\lambda+1\leq l\leq n,$$
which is a polynomial in  variables described  above, with  coefficients rational  on $M$ and  holomorphic on $U_\lambda.$
It yields that
$$\omega(J_{k,\mathfrak{X}}f)\Big{|}_{f^{-1}(U_{\lambda})}=P_\lambda\Big{(}\frac{\mathfrak{X}^i(w_{\lambda j}\circ f)}{w_{\lambda j}\circ f},
\ \mathfrak{X}^h(w_{\lambda l}\circ f)\Big{)}.$$
Let $\{c_\lambda\}$ be a partition of unity subordinate to $\{U_\lambda\},$ which implies that
$$\omega(J_{k,\mathfrak{X}}f)=\sum_{\lambda}c_\lambda\circ f\cdot P_\lambda\Big(\frac{\mathfrak{X}^i(w_{\lambda j}\circ f)}{w_{\lambda j}\circ f},
\ \mathfrak{X}^h(w_{\lambda l}\circ f)\Big)$$
on $M.$ Since $P_\lambda$ is a polynomial, then
\begin{eqnarray*}
  &&\log^+\big{|}\omega(J_{k,\mathfrak{X}}f)\big{|} \\
  &\leq& O\Big{(}\sum_{i,j,\lambda}\log^+\Big|\frac{\mathfrak{X}^i(w_{\lambda j}\circ f)}{w_{\lambda j}\circ f}\Big|
  +\sum_{h,l,\lambda}\log^+\Big{(}c_\lambda\circ f\cdot\Big|\mathfrak{X}^h(w_{\lambda l}\circ f)\Big|\Big{)}\Big{)}+O(1).
\end{eqnarray*}
Note that $L_A>0,$  by virtue of Theorem \ref{ldl2} and Proposition \ref{v000}
\begin{eqnarray*}
&& m\big{(}r,\omega(J_{k,\mathfrak{X}}f)\big{)} \\
&\leq& O\Big{(}\sum_{i,j,\lambda}m\Big{(}r, \frac{\mathfrak{X}^i(w_{\lambda j}\circ f)}{w_{\lambda j}\circ f}\Big{)}+
\sum_{h,l,\lambda}m\Big{(}r, c_\lambda\circ f\cdot \mathfrak{X}^h(w_{\lambda l}\circ f)\Big{)}\Big{)}+O(1) \\
&\leq& O\Big{(}\log T_{f,A}(r)-\kappa(r)r^2+\log^+\log r\Big{)} \\
&&+O\Big{(}\sum_{h,l,\lambda}m\Big{(}r, c_\lambda\circ f\cdot \mathfrak{X}^h(w_{\lambda l}\circ f)\Big{)}\Big{)}.
\end{eqnarray*}
For the last term in the above inequality, we claim that 
\begin{eqnarray*}
 m\Big{(}r, c_\lambda\circ f\cdot \mathfrak{X}^h(w_{\lambda l}\circ f)\Big{)} 
 &\leq& O\Big{(}\log T_{f,A}(r)-\kappa(r)r^2+\log^+\log r\Big{)}.
\end{eqnarray*}
Equip $L_A$ with a Hermitian metric $h$ such that $c_1(L_A,h)>0$ and set
\begin{equation}\label{w1}
  f^*c_1(L_A,h)=\xi\alpha, \ \ \ \alpha=g\frac{\sqrt{-1}}{2\pi}dz\wedge d\bar{z}
\end{equation}
in a local holomorphic coordinate $z.$  $c_1(L_A,h)>0$ implies that there exists a constant  $B>0$ such that
\begin{equation}\label{w2}
  c_\lambda^2\circ f \cdot \frac{\sqrt{-1}}{2\pi}dw_{\lambda l}\circ f\wedge d\overline{w}_{\lambda l}\circ f\leq B f^*c_1(L_A,h).
\end{equation}
Write $\mathfrak{X}=a\frac{\partial}{\partial z},$ then $\|\mathfrak{X}\|^2=g|a|^2.$ Combine (\ref{w1}) with (\ref{w2}),  we have
$$c_\lambda^2\circ f\cdot \big{|}\mathfrak{X}(w_{\lambda l}\circ f)\big{|}^2\leq B\xi\|\mathfrak{X}\|^2.$$
By this with (\ref{ok}) and Proposition \ref{yyyy}
\begin{eqnarray*}
&&m\Big{(}r, c_\lambda\circ f\cdot \mathfrak{X}(w_{\lambda l}\circ f)\Big{)} \\
&\leq& \frac{1}{2} \mathbb E_o\big[\log^+\big{(}B\xi(X_{\tau_r})\|\mathfrak{X}(X_{\tau_r})\|^2\big)\big] \\
&\leq& \frac{1}{2}\mathbb E_o\big[\log^+\xi(X_{\tau_r})\big]+
\frac{1}{2}\mathbb E_o\left[\log^+\|\mathfrak{X}(X_{\tau_r})\|^2\right]+O(1) \\
&\leq& \frac{1}{2}\mathbb E_o\big[\log^+\xi(X_{\tau_r})\big]-\kappa(r)\mathbb E_o\big[\tau_r\big]+O(1) \\
&\leq& \frac{1}{2}\mathbb E_o\big[\log^+\xi(X_{\tau_r})\big]-\frac{\kappa(r)r^2}{2}+O(1).
\end{eqnarray*}
Indeed, it follows from Theorem \ref{cal} that
\begin{eqnarray*}
\mathbb E_o\big[\log^+\xi(X_{\tau_r})\big]
&\leq& \log^+\mathbb E_o\big[\xi(X_{\tau_r})\big]+O(1) \\
&\leq& \log^+\mathbb E_o\left [\int_0^{\tau_r}\xi(X_t)dt\right] \\
&&+
O\Big(\log^+\log \mathbb E_o\left[\int_0^{\tau_{r}}\xi(X_{t})dt\right]+r\sqrt{-\kappa(r)}+\log^+\log r\Big) \\
&\leq& \log T_{f,A}(r)+O\Big(\log^+\log T_{f,A}(r)+r\sqrt{-\kappa(r)}+\log^+\log r\Big)
\end{eqnarray*}
due to
\begin{eqnarray*}
\log^+\mathbb E_o\left [\int_0^{\tau_r}\xi(X_t)dt\right]&=& \log^+\int_{D(r)}g_r(o,x)\frac{ f^*c_1(L_A,h)}{\alpha}dV(x) \\
&\leq& \log T_{f,A}(r)+O(1).
\end{eqnarray*}
Hence,  for the case $h=1$
$$  m\Big{(}r, c_\lambda\circ f\cdot \mathfrak{X}(w_{\lambda l}\circ f)\Big{)}\leq O\Big{(}\log T_{f,A}(r)-\kappa(r)r^2+\log^+\log r\Big{)}.
$$
For the general case, we use mathematical induction.  
Assume that the claim holds for the case $h=k-1,$ then for  $h=k$
\begin{eqnarray*}
&&m\Big{(}r, c_\lambda\circ f\cdot \mathfrak{X}^k(w_{\lambda l}\circ f)\Big{)} \\
&\leq& m\Big{(}r, \frac {\mathfrak{X}^{k}(w_{\lambda l}\circ f}{\mathfrak{X}^{k-1}(w_{\lambda l}\circ f})\Big{)}+
m\Big{(}r, c_\lambda\circ f\cdot \mathfrak{X}^{k-1}(w_{\lambda l}\circ f)\Big{)} \\
&\leq& O\Big{(}\log T_{f,A}(r)+\log^+\log T_{f,A}(r)-\kappa(r)r^2+\log^+\log r\Big{)}
\end{eqnarray*}
due to Proposition \ref{hello1}.
Thus, the claim is confirmed. Combining the above, This proves the first assertion, and the second assertion is proved similarly   by replacing $\kappa(r)r^2$ by $-r.$
\end{proof}
\begin{cor} Let $f:\mathbb C\rightarrow M$ be a holomorphic curve
such that $f(\mathbb C)\not\subseteq D,$ where $D$ is a reduced divisor on $M.$  Then
$$m\big{(}r,\omega(J_{k,\mathfrak{X}}f)\big{)}\leq O\Big{(}\log T_{f,A}(r)+\log^+\log r\Big{)}  \big{\|}$$
for  $\omega\in H^0(M, E_{k,m}^{GG}T^*_M(\log D)).$
\end{cor}

\subsection{Proof of Theorem \ref{thm2}}~
\begin{proof} By lifting $f$ to the universal covering, we may assume that $S$ is simply connected.
Since $A$ is ample, then there is an integer $n_0>0$ such that $n_0A$ is very ample. Hence, $L_{n_0A}$ induces a holomorphic embedding of $M$ into
a complex projective space.
Viewing $M$ as an algebraic submanifold of $\mathbb P^N(\mathbb C)$ with homogeneous coordinate system  $[w_0;\cdots:w_N],$
 i.e., there is an inclusion $i:M\hookrightarrow\mathbb P^N(\mathbb C)$ for some  $N.$
Let $E$ be the restriction  of the hyperplane line bundle over $\mathbb P^N(\mathbb C)$  to $M.$
By virtue of Theorem \ref{thm}, for any $0<\epsilon<1,$ there exists a hyperplane $H$ of $\mathbb P^N(\mathbb C)$ such that
$$(1-\epsilon)T_{f,E}(r)\leq N_{f}(r, i^*H)+O\Big(\log T_{f,E}(r)-\kappa(r)r^2+\log^+\log r\Big).$$
\ \ \ \ By replacing $\omega$ by $(\omega/\|s_A\|)^{n_0}i^*\hat{H},$ where $s_A$ is the canonical line bundle associated to $A$ and $\hat{H}$ is the defining
function of $H,$
we may assume without loss of generality that  $A=i^*H$ and $n_0=1.$  Namely,  we have
$$(1-\epsilon)T_{f,A}(r)\leq N_{f}(r,A)+O\Big(\log T_{f,A}(r)-\kappa(r)r^2+\log^+\log r\Big).$$
By the assumption that $\kappa(r)r^2/T_{f,A}(r)\rightarrow0$ as $r\rightarrow\infty,$ it yields that 
\begin{equation}\label{dfd}
  (1-\epsilon)T_{f,A}(r)\leq  N_{f}(r,A)+o\big(T_{f,A}(r)\big).
\end{equation}
If $f^*\omega\not\equiv0$ on $S,$ then the condition  $\omega|_A\equiv 0$ implies that
$$ f^*A\leq dd^c \big[\log |\omega(J_{k,\mathfrak{X}}f)|^2\big]$$
in the sense of currents. By Dynkin formula and Theorem \ref{xxx}
$$N_f(r,A)\leq m\big{(}r,\omega(J_{k,\mathfrak{X}}f)\big{)}+O(1)\leq o\big(T_{f,A}(r)\big),$$
which contradicts with (\ref{dfd}). This proves the first assertion,  and the second assertion  is proved  similarly by replacing $\kappa(r)r^2$ by $-r.$
\end{proof}

\section{Bloch's theorem}

In this section, we extend  Bloch's theorem \cite{bloch}, following P${\rm{\check{a}}}$un-Sibony \cite{Sibony-Paun} and Ru-Sibony \cite{Ru-Sibony}.

Let $T$ be a complex torus of complex dimension $n,$  fix a  nowhere-vanishing holomorphic vector field $\frak X$ over
$S$. Then the $k$-jet bundle over $T$ with respect to $\frak X$ is trivial, i.e.,
$J_k(T,\frak X)\cong T\times \CC^{nk}$.  Let $P(J_k(T,\frak X))$ be the
projective bundle of $J_k(T,\frak X)$, then we have
$P(J_k(T,\frak X))\cong T\times \PP^{nk-1}(\CC)$. Let $X$ denote the Zariski
closure of $f(S),$ and $X_k\subseteq P(J_k(T,\frak X))$ denote the
Zariski closure of $J_{k,\frak X}f(S).$  Let $\pi_k:X_k\to \PP^{nk-1}(\CC)$ be the
second projection.

\begin{lemma}[Ueno, \cite{Ueno}]\label{thm:ueno}
Let $V$ be a subvariety of  $T$. There exists a
complex torus $T_1\subseteq T$, a projective
variety $W$ and an Abelian variety $A$ such that

${\rm{(i)}}$ $W\subseteq A$ and $W$ is of general type$;$

${\rm{(ii)}}$ There exists a reduction mapping $R:V\to W$ such that whose general
fiber is isomorphic to $T_1$.
\end{lemma}

\begin{lemma}[Sibony-P${\rm{\check{a}}}$un, \cite{Sibony-Paun}]\label{lem:+dimfiber}
 Assume that for each $k\geq 1$,
the fibers of $\pi_k$ are  positive dimensional. Then the subgroup $A_X$
of $X$ defined by
$$A_X:=\big\{a\in T:a+X=X\big\}$$
is  strictly positive dimensional.
\end{lemma}
\begin{lemma}[Ru-Sibony, \cite{Ru-Sibony}]\label{prop:existjet}
Assume that $\pi_k:X_k\to \PP^{nk-1}(\CC)$ has finite generic fibers for some $k>0$. Then there is a
$k$-jet differential $\omega$ with values in the dual of an ample line bundle such that $\omega\not\equiv0$ on $X_k$.
\end{lemma}
Now we prove a generalized Bloch's theorem, i.e., Theorem \ref{thm4}:
\begin{proof}
Assume that $X$ is not the translate of a subtorus.
Let $R:X\to W$ be the reduction mapping in Lemma \ref{thm:ueno}. Then $W$
is not a point since otherwise $X$ would be the translation of a subtorus.
Now we claim that $R\circ f$ cannot satisfy (\ref{th3}).
If not, then there will be  two possible cases:
\begin{enumerate}
\item For any $k\geq 1$, the fibers of $\pi_k$ are positive dimensional;
\item There exists a $k$ such that $\pi_k:X_k\to \PP^{mk-1}(\CC)$ is
generically finite.
\end{enumerate}

If the case (i) happens, then from Lemma \ref{lem:+dimfiber}, $X$ is
stabilized  by a positive-dimensional subtorus of $T$. But we assumed that $X$
is of general type, then it cannot happen since the automorphism group
of $X$ must be of finite order.

If the case (ii) happens, then $X_k$ is algebraic.
By Lemma \ref{prop:existjet},
there exists   a nonzero $k$-jet differential $\omega\in H^0(X_k, E_{k,m}^{GG}T^*_{X_k}\otimes\mathscr O(-A))$
for some ample line bundle $A$, but  Theorem \ref{thm2} gives a contradiction. So,  (ii) cannot
happen.
 We have arrived at a contradiction in both cases. Thus, $R\circ f$ cannot satisfy
(\ref{th3}).
\end{proof}

\vskip\baselineskip

\label{lastpage-01}

\vskip\baselineskip
\vskip\baselineskip

\end{document}